\documentclass[12pt,a4paper,reqno]{amsart}
\usepackage{amsmath}
\usepackage{amsfonts}
\usepackage{amssymb}
\usepackage[dvipsnames]{xcolor}
\usepackage{mathrsfs}
\usepackage{cleveref}
\usepackage{enumitem}
\usepackage{cite}
\usepackage[foot]{amsaddr}
\newtheorem{theorem}{Theorem}[section]
\newtheorem{corollary}[theorem]{Corollary}
\newtheorem{lemma}[theorem]{Lemma}
\newtheorem{proposition}[theorem]{Proposition}
\newtheorem{definition}[theorem]{Definition}
\newtheorem{remark}[theorem]{Remark}
\newtheorem{assumption}[theorem]{Assumption}
\numberwithin{equation}{section}
\begin{document}
	
\title{Well-posedness of the growth-coagulation equation with singular kernels} 
%\thanks{}
\author{Ankik Kumar Giri $^{1,\ast}$}
\address{$^1$ Department of Mathematics, Indian Institute of Technology Roorkee, Roorkee-247667, Uttarakhand, India}
\address{$\ast$ Corresponding author. Tel +91-1332-284818 (O);  Fax: +91-1332-273560}
\email{ankik.giri@ma.iitr.ac.in}	

\author{Philippe Lauren\c{c}ot $^2$}
\address{$^2$ Laboratoire de Math\'ematiques (LAMA) UMR~5127, Universit\'e Savoie Mont Blanc, CNRS, F--73000 Chamb\'ery, France}
\email{philippe.laurencot@univ-smb.fr}
\author{Saroj Si $^1$}
\email{ssaroj@ma.iitr.ac.in}


\date{\today}

%%%%%%%%%%%%%%%%
%%%%%%%%%%%%%%%%
\begin{abstract}
	The well-posedness of the growth-coagulation equation is established for coagulation kernels having singularity near the origin and growing atmost linearly at infinity. The existence of weak solutions is shown by means of the method of the characteristics and a weak $L_1$-compactness argument. For the existence result, we also show our gratitude to Banach fixed point theorem and a refined version of the Arzel\'{a}-Ascoli theorem. In addition, the continuous dependence of solutions upon the initial data is shown with the help of the DiPerna-Lions theory, Gronwall's inequality and moment estimates. Moreover, the uniqueness of solution follows from the continuous dependence. The results presented in this article extend the contributions made in earlier literature. 
\end{abstract}
%%%%%%%%%%%%%%%%
%%%%%%%%%%%%%%%%

\maketitle

%%%%%%%%%%%%%%%%
%%%%%%%%%%%%%%%%
\section{introduction}
%%%%%%%%%%%%%%%%
%%%%%%%%%%%%%%%%

Coagulation is the process of aggregation where two or more smaller particles or clusters combine to form a larger particle or cluster.  In nature, this process is found in phytoplankton aggregation \cite{Adler1997}, polymer formation \cite{Ziff1980}, powder formation in industry \cite{Wells2018}, droplet formation in clouds \cite{Friedlander2000} and formation of planets in astrophysics \cite{Lissauer1993, Safronov1972}. These physical processes are modelled mathematically by partial integro-differential equations which can be read as
\begin{align}
	\partial_{t}c(t,v) &= 	Q(c)(t,v)  \quad \text{for} \; (t,v) \in (0,\infty)^2,\label{eq:1.1} \\
	c(0,v)&=c_0(v) \geq 0,\label{eq:1.2}
\end{align}
where 
\begin{align*}
	Q(c)(v_{1})&= \frac{1}{2} \int_{0}^{v_{1}} K(v_{1}-v_{2},v_{2})c(v_{1}-v_{2})c(v_{2})dv_{2} \nonumber \\ 
	\;\;&\qquad -\int_{0}^{\infty} K(v_{1},v_{2})c(v_{1})c(v_{2})dv_{2}.  
\end{align*}
Here, $K(v_{1},v_{2})$ denotes the coagulation kernel which gives the rate at which particles of sizes $v_{1}$ and $v_{2}$ merge to form larger particles of size $v_{1}+v_{2}$ and is a non-negative \emph{symmetric function} on $(0,\infty)^{2}$. The function $c(t,v_{1})$ represents the density of particles with size $v_{1}>0$ at time $t\geq 0$. The first integral in the definition of $Q(c)$ describes the formation of particles of size $v_{1}$ due to the interaction between the particles of sizes $v_{1}-v_{2}$ and $v_{2}$. On the other hand, the second integral in the definition of $Q(c)$ accounts for the loss of particles of size $v_{1}$ due to their merging with particles of arbitrary size.
A great amount of progress in this direction has been made over the last hundred years, see for example \cite{BLL_book} and the references therein. There is one more kinetic process which can be associated with coagulation, termed as growth process. The growth associated with coagulation is governed by the following mathematical equation
\begin{align}	
	\partial_{t}c(t,v) + \partial_{v}(gc)(t,v) &= 	Q(c)(t,v)  \quad \text{for} \; (t,v) \in (0,\infty)^2, \label{eq:1.3}\\
	c(0,v)&=c_0(v) \geq 0, \label{eq:1.4}  
\end{align}
where $g(t,v)\ge 0$ represents the rate at which external matter is sticking to the surface of the particles of size $v$, see \cite{AcklehDeng2003, BLL_book, Berry1969, Gajewski1983, GajewskiZacharias1982, LevinSedunov1968,  LushnikovKulmala2000} and the references therein, so that the second term on the left-hand side of~\eqref{eq:1.3} represents the growth of the  particles of size $v_{1}$. In contrast to the Smoluchowski coagulation equation~\eqref{eq:1.1}, the initial value problem~\eqref{eq:1.3}--\eqref{eq:1.4} has received less attention in the mathematical literature. As far as the existence of a solution to \eqref{eq:1.3}--\eqref{eq:1.4} is concerned, the initial investigation is conducted by Gajewski and Zacharias \cite{GajewskiZacharias1982}, where existence is demonstrated in the Hilbert space $L^2$ for $g\in W^{2,\infty}(0,\infty)$ and the class of coagulation kernels satisfying
	\begin{equation}
		0\leq K(v_{1},v_{2}) = K(v_2,v_1) \leq C_{0}\left(1+(v_{1}v_{2}\right)^{-\beta_0}), \quad (v_{1}, v_{2})\in (0,\infty)^{2}, \label{ge_1}
	\end{equation}
for some $\beta_0\ge 0$ and $C_0>0$ and
\begin{equation*}
 -C_{1} \left(1+\frac{1}{v_{1}}\right)K(v_{1},v_{2})\leq \partial_{v_{1}}K (v_{1},v_{2})\leq 0,  \quad (v_{1}, v_{2})\in (0,\infty)^{2},
\end{equation*}
for some $C_1>0$. In a subsequent study \cite{Gajewski1983}, Gajewski extends the existence result obtained in \cite{GajewskiZacharias1982} to the class of coagulation kernels which satisfies only~\eqref{ge_1} using a semigroup theory approach in the Hilbert space $L^{2}$ and also includes an additional source term. However, the contributions in \cite{GajewskiZacharias1982, Gajewski1983} do not cover the physically relevant weighted $L^1$-space and coagulation kernels which are unbounded for large sizes. Global existence of solutions to \eqref{eq:1.3}--\eqref{eq:1.4} and uniqueness are studied in \cite{AcklehDeng2003, Banasiak2012} for bounded coagulation kernels and in \cite{FriedmanReitich1990} for the class of unbounded coagulation kernels having the following growth condition,
\begin{equation*}
	K(v_{1},v_{2})\leq \sigma(v_{1}+v_{2}),\quad (v_{1}, v_{2})\in [1,\infty)^{2},
\end{equation*} 
where $\sigma$ is a sublinear function at infinity; that is, $\sigma(v)/v\to 0$ as $v\to\infty$. Related models involving growth and multidimensional coagulation have been studied in \cite{CristianVelazquez, HingantSepulveda2015}.

The primary focus of this article is to investigate the existence and uniqueness to \eqref{eq:1.3}--\eqref{eq:1.4} in the weighted $L^{1}$ space for the class of coagulation kernels that exhibit singularities for small  sizes while being unbounded for large ones and is mainly motivated by \cite{AcklehDeng2003, ColletGoudon1999, FriedmanReitich1990, Gajewski1983, Laurencot2001}, with the aim of including a broader class of coagulation kernels in the forthcoming analysis. 

The structure of the paper unfolds as follows: In \Cref{Assumption and Main Result}, we outline the assumptions on coagulation kernels $K$, growth rate $g$, and initial data $c_0$. This section also introduces the solution concept and highlights the main theorems of the paper. Moving on to \Cref{Characteristics}, we delve into characteristics and the mild formulation of our problem~\eqref{eq:1.3}--\eqref{eq:1.4}. Dedicated to demonstrating the existence of global weak solutions for \eqref{eq:1.3}--\eqref{eq:1.4}, \Cref{Global existence of weak solutions} takes center stage. Firstly, we define the truncated problem \eqref{truncated-1}--\eqref{truncated-2}. Subsequently, we establish the existence of a unique global mild solution to \eqref{truncated-1}--\eqref{truncated-2} using the Banach fixed-point theorem, following the methodology in \cite{ColletGoudon1999}. This section also encompasses the derivation of moment estimates for the mild solution, building on the insights from \cite{Laurencot2001} and relying on the method of weak $L_1$-compactness introduced in \cite{Stewart1989}. Finally, in the last section, we turn our attention to demonstrating the uniqueness and continuous dependence on the initial data of solutions to~\eqref{eq:1.3}--\eqref{eq:1.4}.

%%%%%%%%%%%%%%%%
%%%%%%%%%%%%%%%%
\section{ assumption and main result}\label{Assumption and Main Result}
%%%%%%%%%%%%%%%%
%%%%%%%%%%%%%%%%

%%%%%%%%%%%%%%%%
\begin{assumption}\label{Assumption}
	Throughout this paper, we impose the following assumptions.
	
	\begin{enumerate}[label=(\roman*)]
		\item 	  $K$ is a non-negative measurable symmetric function on $(0,\infty)^2.$
		There are $\beta > 0$  and $k,k_{1}>0$ such that
		
		\begin{align}
			K(v_{1},v_{2}) = K(v_{2},v_{1}) \leq
			\left\{
			\begin{aligned} \label{coagassum}
				&k(v_{1}v_{2})^{-\beta},\quad(v_{1},v_{2})\in (0,1)^2,\\   
				&kv_{2}v_{1}^{-\beta},\qquad(v_{1},v_{2})\in (0,1) \times (1,\infty),\\
				&k(v_{1}+v_{2}),\quad(v_{1},v_{2})\in (1,\infty)^2,
			\end{aligned} 
			\right.
		\end{align}
	
	\item $g$ is a non-negative continuous function on $[0,\infty)^2$ and twice differentiable with respect to $v_{1}$ with
	\begin{align}
		g(t,0)=0 \text{ for all}\quad t\geq 0, \quad |\partial_{v_{1}} g| < A, \quad |\partial_{v_{1}}^2g| < B,	\label{growthassum}
	\end{align}
	where $A$ and $B$ are positive real numbers.
	
		\item The initial condition $c_0$ satisfies the following condition
	\begin{align}
		0\leq c_0 \in L_{-2\beta,1}^{1}(0,\infty)  := L^1\big(0,\infty;\big(v+v^{-2\beta}\big)dv\big).\label{initialassum}
	\end{align} 
	
\end{enumerate}

\end{assumption}
%%%%%%%%%%%%%%%%

For further use we set $L_{m,r}^{1}(0,\infty):=L^1\big(0,\infty;\big(v^{m}+v^{r}\big)dv\big)$ for $(m,r)\in\mathbb{R}^2$ and $L_m^1(0,\infty):=L_{m,m}^1(0,\infty)$,  and define the moment of order $m$ by
\begin{equation*}
	M_m(u) := \int_0^\infty v^m u(v)\,dv, \quad u\in L_m^1(0,\infty).
\end{equation*}.
We also denote the positive cone of $L_{m,r}^{1}(0,\infty)$ by $L_{m,r,+}^{1}(0,\infty)$.

Now, let us discuss a class of physically relevant unbounded coagulation kernels that satisfy~\eqref{coagassum}. This class includes Smoluchowski's coagulation kernel $K_1$ \cite{Smoluchowski1917}, the Granulation kernel $K_2$ \cite{Kapur1972}, and stochastic stirred froths $K_3$ \cite{ClarkKatsouros1999}. These kernels take the following forms
\begin{equation*}%\label{K_1}
	K_1(v_{1},v_{2}) = \left({v_{1}}^{\frac{1}{3}}+{v_{2}}^{\frac{1}{3}}\right)\left({v_{1}}^{-\frac{1}{3}}+{v_{2}}^{-\frac{1}{3}}\right),
\end{equation*}
\begin{equation*}%\label{K_2}
	K_2(v_{1},v_{2}) = \frac{(v_{1}+v_{2})^{\theta_{1}}}{(v_{1}v_{2})^{\theta_{2}}},\quad \theta_{1}\leq 1 \;\text{and} \; \theta_{2}\geq 0,
\end{equation*}
\begin{equation*}%\label{K_3}
	K_3(v_{1},v_{2}) = \left({v_{1}v_2}\right)^{-\theta}, \quad \theta>0,
\end{equation*}
for $(v_1,v_2)\in (0,\infty)^2$ and satisfy~\eqref{coagassum} with $\beta=\frac{1}{3}$, $\beta=\theta_2$, and $\beta=\theta$, respectively.

Before presenting our main results, we provide the definition of a weak solution to \eqref{eq:1.3}--\eqref{eq:1.4} that we will use in the subsequent sections.

%%%%%%%%%%%%%%%%
\begin{definition} \label{definitionofsolution}
	Let $K$ and $g$ be two functions satisfying~\eqref{coagassum}--\eqref{growthassum} and $c_0$ be an initial condition satisfying~\eqref{initialassum}. For $T>0$, a weak solution to~\eqref{eq:1.3}--\eqref{eq:1.4} on $[0,T]$ is a non-negative function 
	\begin{equation}
		c \in C\left([0,T];w-L_{-\beta,1}^1(0,\infty)\right) \label{weak formulation-0} 
	\end{equation}
	satisfying, for each $t \in [0, T]$ and $\phi \in W^{1,\infty}\big(0,\infty\big)$,
	\begin{equation} \label{weak formulation}
		\begin{split}
		\int_{0}^{\infty} c(t,v)\phi(v)dv  & =   	\int_{0}^{\infty} c_0(v)\phi(v)dv   
		+\int_{0}^{t}\int_{0}^{\infty} \phi^{\prime}(v)g(s,v)c(s,v)dvds \\
		&+\frac{1}{2}\int_{0}^{t}\int_{0}^{\infty} \int_{0}^{\infty}\tilde{\phi}(v_{1},v_{2}) K(v_{1},v_{2})c(s,v_{2}) c(s,v_{1})dv_{1} dv_{2}ds,
		\end{split}
	\end{equation}
where 
\begin{equation*}
	\tilde{\phi}(v_{1},v_{2}):=\phi(v_{1}+v_{2})-\phi(v_{1})-\phi(v_{2}).
\end{equation*}
We shall say that a weak solution to~\eqref{eq:1.3}--\eqref{eq:1.4} is global if it is a weak solution to~\eqref{eq:1.3}--\eqref{eq:1.4} on $[0,T]$ for all $T>0$.
\end{definition}
%%%%%%%%%%%%%%%%

Let us mention at this point that, if $\phi \in W^{1,\infty}(0,\infty)$, then
\begin{align*} 
	\left|\phi^{\prime}(v_1) g(t,v_1)\right| & \leq g(t,v_1)\|\phi\|_{ W^{1,\infty}(0,\infty)} \leq A \|\phi\|_{ W^{1,\infty}(0,\infty)} v_1, \\
	\tilde{\phi}(v_1,v_2) K(v_1,v_2) & \le 3 \|\phi\|_{L^\infty(0,\infty)} \left( v_1 + v_1^{-\beta} \right) \left( v_2 + v_2^{-\beta} \right),
\end{align*}
for $t\in [0,T)$ and $(v_1,v_2)\in (0,\infty)^3$ by~\eqref{coagassum} and~\eqref{growthassum}, so that, thanks to \eqref{weak formulation-0}, all the integrals in \eqref{weak formulation} are well-defined.

\medskip
	
Let us now state the main results of this article.

%%%%%%%%%%%%%%%%
\begin{theorem} \label{main theorem}
	Suppose that $K$, $g$, and $c_0$ satisfy~\eqref{coagassum}, \eqref{growthassum} and~\eqref{initialassum}, respectively. Then there exists at least one global weak solution $c$ to~\eqref{eq:1.3}--\eqref{eq:1.4} in the sense of \Cref{definitionofsolution} satisfying also
	\begin{equation*}
		c \in L^{\infty}\big(0,T;L_{-2\beta,1}^1(0,\infty)\big) \;\;\text{ for all }\;\; T>0.
	\end{equation*}	
\end{theorem}
%%%%%%%%%%%%%%%%

The proof of \Cref{main theorem} utilizes a method of characteristics introduced in \cite{ColletGoudon1999} and subsequently employs a weak compactness approach within the space $L_{-2\beta,1}^1(0,\infty)$, a methodology initially introduced in \cite{Stewart1989} for the coagulation-fragmentation equation and later adapted in \cite{Laurencot2001} for coagulation kernels $K$ growing at most linearly for large sizes; that is, there is $C>0$ such that $K(v_1,v_2) \le C(1+v_1+v_2)$ for $(v_1,v_2)\in (0,\infty)^2$.

\medskip

In addition to establishing the existence of weak solutions, a result on uniqueness is proved for a limited class of initial data $c_{0}$.

%%%%%%%%%%%%%%%%
\begin{theorem} \label{uniqueness theorem}
	Suppose that $K$ and $g$ satisfy~\eqref{coagassum} and~\eqref{growthassum}, and consider $c_0 \in L_{-2\beta,2,+}^1(0,\infty)$. Then there exists a unique global weak solution $c$ to~\eqref{eq:1.3}--\eqref{eq:1.4}  in the sense of  \Cref{definitionofsolution} such that
	\begin{equation}\label{solsecondmomentbound}
		c\in L^{\infty}\big(0,T;L_{-2\beta,2}^1(0,\infty)\big) \;\;\text{ for all }\;\; T>0.
	\end{equation}
\end{theorem}
%%%%%%%%%%%%%%%%

To establish uniqueness, akin to \cite{EMRR2005, Stewart1990}, we demonstrate a more generalized continuous dependence result in \Cref{continuousdependence result}. The complexity of the proof stems from the low regularity of weak solutions and the task of identifying an appropriate weight function, and in this instance, it is represented by $w(v) := v^{-\beta} + v$, $v>0$.

%%%%%%%%%%%%%%%%
\begin{remark}
	When $g\equiv 0$,  $c_0 \in L_{-2\beta,2}^1(0,\infty)$ and K satisfies \eqref{coagassum}, we recover from \Cref{uniqueness theorem} the well-posedness of the continuous coagulation equation \eqref{eq:1.1}--\eqref{eq:1.2} already established in \cite{BarikGiriLaurencot2019, CamejoGroplerWarnecke2015, CamejoWarnecke2015}.
\end{remark}
%%%%%%%%%%%%%%%%

%%%%%%%%%%%%%%%%
%%%%%%%%%%%%%%%%
\section{Characteristic curve and mild solution}\label{Characteristics}
%%%%%%%%%%%%%%%%
%%%%%%%%%%%%%%%%

%%%%%%%%%%%%%%%%
%%%%%%%%%%%%%%%%
\subsection {Characteristic curves}
%%%%%%%%%%%%%%%%
%%%%%%%%%%%%%%%%

In this section, first we define the characteristics, and using the characteristic curves, we define the mild solution to \eqref{eq:1.3}--\eqref{eq:1.4}.

\medskip

To begin with, let us observe that \eqref{growthassum} implies that
\begin{align}
	0\le g(t,v)&\leq A v  & \; \text{for}\; (t,v)\in [0,\infty)^2 ,\label{characteristic-1}\\
	|g(t,v_{1})-g(t,v_{2})|& \leq A |v_{1}-v_{2}| & \text{ for }\;  (t,v_{1},v_{2})\in  [0,\infty)^3. \label{characteristic-2}
\end{align}
Property~\eqref{characteristic-2} ensures that, for $(t,v)\in [0,\infty)^2$, the following characteristics equation is well-posed and has a unique solution $Y(\cdot;t,v)$ defined for $s\in [0,\infty)$
\begin{align}\label{characteristic-3}
	\left\{
	\begin{aligned} 
		\partial_s Y(s;t,v)&=g(s,Y(s;t,v)), \\   
		Y(t;t,v)&=v.
	\end{aligned} 
	\right.
\end{align}

Let us now turn to the following fundamental properties of the characteristic curves, which are taken from \cite{ColletGoudon1999, ColletGoudon2000}.

%%%%%%%%%%%%%%%%
\begin{lemma}
	For $(s,t,v)\in [0,\infty )^3$, the characteristics curves $Y$ satisfy
	\begin{align}
	&Y(t;s,Y(s;t,v))=v, \label{characteristic-4}\\
		&J(s;t,v):= \partial_{v}Y(s;t,v) = \exp \left(\int_{s}^{t}b\left(\sigma, Y(\sigma;t,v)\right)d\sigma\right), \nonumber\\ %\label{characteristic-5}\\
		&\partial_t Y(s;t,v) = -g(t,v) 
		\; \exp \left(\int_{s}^{t} b\left(\sigma, Y(\sigma;t,v)\right)d\sigma\right), \nonumber\\ %\label{characteristic-6} \\
&Y(t;s,0)=0 , \nonumber\\ %\label{characteristic-7}\\
&Y(t;s,v)>0 \; \text{for}\; v>0. \nonumber %\label{characteristic-8}
\end{align}
where $b:=-\partial_{v} g$.\\
\end{lemma}
%%%%%%%%%%%%%%%%

The proof of this lemma can be seen in \cite[Appendix~A]{ColletGoudon2000}. We next discuss some essential properties of solutions to~\eqref{characteristic-3} which are frequently used in this paper.

%%%%%%%%%%%%%%%%
\begin{lemma}
	For $(t,v)\in  [0,\infty)\times (0,\infty )$ then the characteristic curve $Y(\cdot;t,v)$ solving \eqref{characteristic-3} satisfies the following properties:
	\begin{align}
		&for\; s_1\leq s_2\leq t, & 0< Y(s_2;t,v) \leq Y(s_1;t,v) e^{A(s_2-s_1)}, \label{inequality-8}\\
		& & \lim_{v \rightarrow \infty} Y(0;t,v) = \infty, \nonumber\\ %\label{inequality-9}\\
		& for\; s\geq t, & Y(s;t,v)\leq ve^{A(s-t)}, \label{inequality-10}\\
		& & Y(t;0,v)\leq ve^{At}, \nonumber\\ %\label{inequality-11}\\
		&for\; s\geq 0, & |g(s,Y(t;0,v))|\leq Ave^{At}, \label{inequality-12}\\
		& for\; s\geq t, & Y(s;t,v) \geq v. \label{inequality-13}
	\end{align}
\end{lemma}
%%%%%%%%%%%%%%%%

\begin{proof}
The detailed justification of the properties~\eqref{inequality-8}-\eqref{inequality-12}   can be found in \cite [Lemma~1]{ColletGoudon2000}. Inequality~\eqref{inequality-13} follows from~\eqref{characteristic-3} and the non-negativity of $g$ on $[0,\infty)^2$.	
\end{proof}

%%%%%%%%%%%%%%%%
%%%%%%%%%%%%%%%%
\subsection{Mild solution}
%%%%%%%%%%%%%%%%
%%%%%%%%%%%%%%%%

Now, by using the characteristic curves, we define the notion of \emph {mild solution} to  \eqref{eq:1.3}--\eqref{eq:1.4}.

%%%%%%%%%%%%%%%%
\begin{definition}\label{defin:mild}
	Let $K$ and $g$ be two functions satisfying~\eqref{coagassum}--\eqref{growthassum} and $c_0$ be an initial condition satisfying~\eqref{initialassum}. For $T>0$, a mild solution to~\eqref{eq:1.3}--\eqref{eq:1.4} on $[0,T)$ is a non-negative function $c\in C([0,T];L^1(0,\infty))$ with $Q(c)\in L^1((0,T)\times (0,\infty))$ which satisfies the mild formulation to \eqref{eq:1.3}--\eqref{eq:1.4}
	\begin{equation*}%\label{mild formulation}
		c(t,v)=c_0(Y(0;t,v))J(0;t,v) + \int_{0}^{t} Q(c)(s,Y(s;t,v)) J(s;t,v)ds,
	\end{equation*}
	for all $t\in [0,T]$ and $v\in (0,\infty)$.
	
	We shall say that a mild solution to~\eqref{eq:1.3}--\eqref{eq:1.4} is global if it is a mild solution to~\eqref{eq:1.3}--\eqref{eq:1.4} on $[0,T]$ for all $T>0$.
\end{definition}
%%%%%%%%%%%%%%%%

In order to prove \Cref{main theorem}, it is necessary to recover the notion of weak solution from that of mild solution. In this connection, we report a relation between mild and weak solutions, see \cite[Proposition~1]{ColletGoudon2000} and \cite[Lemma~2.3]{HingantSepulveda2015}.

%%%%%%%%%%%%%%%%
\begin{proposition}\label{relation of weak and mild solution}
	Let $T>0$ and assume that \eqref{coagassum}--\eqref{initialassum} hold true. The following statements are equivalent:
	\begin{enumerate}[label=(\roman*)]
		\item $c\in C\left([0,T];L^1((0,\infty))\right)$ is a weak solution in the sense of \Cref{definitionofsolution}.
		\item $c$ is a mild solution in the sense of \Cref{defin:mild}.
	\end{enumerate}
\end{proposition}
%%%%%%%%%%%%%%%%

%%%%%%%%%%%%%%%%
%%%%%%%%%%%%%%%%
\section{global existence of weak solutions} \label{Global existence of weak solutions}
%%%%%%%%%%%%%%%%
%%%%%%%%%%%%%%%%

Let $K$, $g$ and $c_0$ satisfy \eqref{coagassum}--\eqref{initialassum} and consider the following truncated problem
\begin{equation}\label{truncated-1}
	\partial_t c_n  +\partial_{v}(gc_n) = Q_n(c_n) \quad \text{for} \; (t,v) \in (0,\infty)^2,
\end{equation}
\begin{equation}\label{truncated-2}
	c_n(0,v)=c_{0}^{n}(v) \ge 0, 
\end{equation}
where, for $(v_1,v_2)\in (0,\infty)^2$,
\begin{align}
	 Q_n(c)(v_1) &:=\frac{1}{2} \int_{0}^{v_{1}} K_n(v_{1}-v_{2},v_{2}) c(v_{1}-v_{2}) c(v_{2}) dv_{2} \nonumber \\
	&\quad	-\int_{0}^{\infty} K_n(v_{1},v_{2}) c(v_{1}) c(t,v_{2}) dv_{2},\nonumber 
\end{align}
and 
\begin{equation*}
		K_n(v_{1},v_{2}):=K(v_{1},v_{2}) \chi_{(1/n,n)}(v_{1}) \chi_{(1/n,n)}(v_{2}), \quad c_{0}^{n}(v_{1}):=c_{0}(v_{1})\chi_{(0,n)}(v_{1}).\nonumber
\end{equation*}
Here, $\chi_E$ denotes the characteristic function of the set $E$; that is, $\chi_E(v)=1$ if $v\in E$ and $\chi_E(v) = 0$ if $v\not\in E$. Observe that the definition of $K_n$ and~\eqref{coagassum} guarantee that $K_n\in L^\infty\big((0,\infty)^2\big)$ and put
\begin{equation*}
	\beta_{n}:=\max_{(v_{1},v_{2})\in (0,\infty)^2}K_{n}(v_{1},v_{2}).
\end{equation*}

Before proving the existence of a mild solution to  \eqref{truncated-1}--\eqref{truncated-2}, we collect some properties of the coagulation operator $Q_{n}$. For $c_{1}\in L^1(0,\infty)$ and $c_{2}\in L^{1}(0,\infty)$, we have
\begin{align}
	\|Q_{n}(c_{1})\|_{L^{1}(0,\infty)} &\leq \frac{3}{2} \beta_{n} \|c_{1}\|_{L^{1}(0,\infty)}^{2}, \label{coagapriori1}\\
	\|Q_{n}(c_{1})-Q_{n}(c_{2})\|_{L^{1}(0,\infty)} & \leq 3 \beta_{n} \left(\|c_{1}\|_{L^{1}(0,\infty)}+\|c_{2}\|_{L^{1}(0,\infty)}\right) \|c_{1}-c_{2}\|_{L^{1}(0,\infty)}. \label{coagapriori2}
\end{align}
Clearly, \eqref{coagapriori1} confirms that, for $c_{1}\in L_+^1(0,\infty)$, $Q_{n}(c_{1})\in L^1(0,\infty)$ with
\begin{equation} \label{negative contribution of Q}
	\int_{0}^{\infty} Q_{n}(c_{1})(v) dv \leq 0.
\end{equation}
Additionally, if $c_{1}\in L_1^{1}(0,\infty)$ then
\begin{equation}
	\int_{0}^{\infty} vQ_{n}(c_{1})(v) dv =0. \label{first moment vanishing} 
\end{equation}

%%%%%%%%%%%%%%%%
\begin{theorem}\label{fixed point theorem-2}
	Let $c_0$ be a non-negative function in $L_{-2\beta,1}^{1}(0,\infty)$. Then the initial value problem~\eqref{truncated-1}--\eqref{truncated-2} has a unique global mild solution in the sense of \Cref{defin:mild}. Moreover, it satisfies
	\begin{equation} \label{zeroth moment}
		\begin{aligned}
			0\leq & \int_{0}^{\infty} c_{n}(t,v)dv \leq \int_{0}^{\infty} c_0(v) dv, \\
			0\leq & \int_{0}^{\infty} v c_{n}(t,v)dv \leq e^{At} \int_{0}^{\infty} v c_0(v) dv,
		\end{aligned}
	\end{equation}
	and
	\begin{equation}
		\int_{0}^{\infty}vc_{n}(t,v) dv=\int_{0}^{\infty}vc_{0}^{n}(v) dv + \int_{0}^{t}\int_{0}^{\infty} g(s,v) c_{n}(s,v)dv ds, \label{rate of first moment}
	\end{equation}
	for all $t\in [0,\infty)$. In addition, $c_n$ is also a global weak solution to~\eqref{truncated-1}--\eqref{truncated-2} and thus satisfies 
	\begin{align} \label{truncated weak formulation}
		\int_{0}^{\infty} c_{n}(t,v)\phi(v)dv  &=   	\int_{0}^{\infty} c_{0}^{n}(v)\phi(v)dv   
		+\int_{0}^{t}\int_{0}^{\infty} \phi^{\prime}(v)g(s,v)c_{n}(s,v)dvds \nonumber \\
		&+\frac{1}{2}\int_{0}^{t}\int_{0}^{\infty} \int_{0}^{\infty}\tilde{\phi}(v_{1},v_{2}) K_{n}(v_{1},v_{2})c_{n}(s,v_{2}) c_{n}(s,v_{1})dv_{1} dv_{2}ds.
	\end{align}
	for all $t\in [0,\infty)$ and $\phi\in W^{1,\infty}(0,\infty)$.
\end{theorem}
%%%%%%%%%%%%%%%%

\begin{proof}
For $n\ge 2$, we set
\begin{equation}\label{smallt_{0}}
	t_{n} := \frac{1}{2\left[ \kappa_n + 6\beta_{n} M_0(c_0) \right]}, \quad \kappa_n :=  \beta_{n} M_0(c_0),
\end{equation}	
and define the set
\begin{equation*}
	P_n := \left\{ u\in C([0,t_{n}];L_+^1(0,\infty))\ :\ \sup_{t\in [0,t_n]}\left\{M_0(u(t))\right\}\leq M_0(c_0),\ M_1(u(t))\leq M_1(c_0) e^{At}\right\},
\end{equation*}
recalling that $A$ is the Lipschitz constant of $g$, see~\eqref{growthassum}, and the moments $M_m$ are defined right after the statement of \Cref{Assumption}. Clearly, $P_n$ is a complete metric space for the distance
\begin{equation*}
	d_{P_n}(u_1,u_2) := \sup_{t\in [0,t_n]}\left\{ \|(u_1-u_2)(t)\|_{L^1(0,\infty)}\right\}, \quad (u_1,u_2)\in P_n^2.
\end{equation*} 
induced by the norm of $C([0,t_{n}];L^1(0,\infty))$.

Now, for $u\in P_n$, we define the map $\mathcal{T}_n(u)$ by
\begin{align*}
	\mathcal{T}_n(u)(t,v):=& c_{0}^{n}(Y(0;t,v))J(0;t,v)  e^{-\kappa_n t}  \\ 
	&\;\;+ \int_{0}^{t} \left(Q_{n}(u)(s,Y(s;t,v)) + \kappa_n u(s,Y(s;t,v))\right) J(s;t,v)e^{-\kappa_n (t-s)} ds 
\end{align*}
for $(t,v)\in [0,t_n]\times (0,\infty)$ and aim at showing the existence of a fixed point of $\mathcal{T}_n$ in $P_n$. Since $\mathcal{T}_n(u)$ is a mild solution to
\begin{equation*}
	\begin{cases}~& {	\partial_{t} U + \partial_{v_{1}}(g U) +\kappa_n  U = 	Q_{n}(u) + \kappa_n   u} \quad \text{for} \; (t,v) \in (0,t_n)\times (0,\infty),\\ 
		~&	U(0,v)=c_{0}^{n}(v) ,
	~\end{cases}  
\end{equation*}
a fixed point of $\mathcal{T}_n$ is clearly a mild solution to~\eqref{truncated-1}--\eqref{truncated-2} on $[0,t_n]$. 

Let us now check that $\mathcal{T}_n$ is a strict contraction on $P_n$. Let $u\in P_n$ and $t\in [0,t_n]$. We first note that, since $u(t)\ge 0$ and $M_0(u(t))\le M_0(c_0)$, we deduce from~\eqref{smallt_{0}} that
\begin{align}
	Q_n(u)(t,v) + \kappa_n  u(t,v) & \ge \kappa_n  u(t,v) - \int_0^\infty K_n(v,v_1) u(t,v) u(t,v_1) dv_1 \nonumber \\
	& \ge \left[ \beta_n M_0(c_0)  - \beta_n M_0(u(t)) \right] u(t,v) \ge 0. \label{global nonnegative-1M}
\end{align} 
Combining this lower bound with the non-negativity of $c_0^n$ and $J$ gives
\begin{equation}
	\mathcal{T}_n(u)\geq 0 \quad \text{for}\; u \in P_n. \label{global nonnegative-2M}
\end{equation}
Then we compute the norm of $\mathcal{T}_n(u)(t)$ in $L^{1}(0,\infty)$ for all $t\in [0,t_{n}]$.
\begin{align*}
	\int_{0}^{\infty}	\mathcal{T}_n(u)(t,v) dv&= 	  e^{-\kappa_n t}	\int_{0}^{\infty} c_{0}^{n}(v)dv  \\
	\quad & + \int_{0}^{t} 	\int_{0}^{\infty}\left(Q_{n}(u)(s,v) + \kappa_n u(s,v)\right) e^{-\kappa_n (t-s)} dvds .  
\end{align*}
By using~\eqref{negative contribution of Q}, we end up with
\begin{equation*}
	M_0\big(\mathcal{T}_n(u)(t)\big) \leq e^{-\kappa_n t} M_0(c_0) + \kappa_n  M_0(c_0) \left( \frac{1-e^{-\kappa_n t}}{\kappa_n }\right)= M_0(c_0)
\end{equation*}
when $\kappa_n>0$ and $M_0\big(\mathcal{T}_n(u)(t)\big) \leq M_0(c_0)$ when $\kappa_n=0$. Therefore, we have established
\begin{equation}
	M_0\big(\mathcal{T}_n(u)(t)\big) \leq M_0(c_0).\label{tau-1}
\end{equation}
We next compute the first moment of $\mathcal{T}_n(u)$ for $u\in P_n$ and find, for $t\in [0,t_n]$, 
\begin{align*}
	\int_{0}^{\infty}	v\mathcal{T}_n(u)(t,v) dv&= e^{-\kappa_n t}	\int_{0}^{\infty} Y(t,0,v) c_{0}^{n}(v)dv  \\
	&\;\; + \int_{0}^{t}\int_{0}^{\infty} \left[Y(t;s,v)\left(Q_{n}(u)(s,v) + \kappa_n u(s,v)\right) e^{-\kappa_n (t-s)}\right] dvds .  
\end{align*}
Owing to~\eqref{inequality-10}, \eqref{first moment vanishing}, and~\eqref{global nonnegative-1M}, we deduce that
\begin{equation*}
	e^{-\kappa_n t} \int_{0}^{\infty} Y(t,0,v) c_{0}^{n}(v)dv \le M_1(c_0) e^{(A-\kappa_n)t}
\end{equation*}
and
\begin{align*}
	& \int_{0}^{\infty} \left[Y(t;s,v)\left(Q_{n}(u)(s,v) + \kappa_n u(s,v)\right) e^{-\kappa_n (t-s)}\right] dv \\
	& \quad \le \int_{0}^{\infty} \left[v\left(Q_{n}(u)(s,v) + \kappa_n u(s,v)\right) e^{(A-\kappa_n )(t-s)}\right] dv \\
	& \quad = \kappa_n M_1(u(s)) e^{(A-\kappa_n )(t-s)} \le \kappa_n M_1(c_0) e^{At - \kappa_n(t-s)}.
\end{align*}
Consequently, 
\begin{equation}
	\int_{0}^{\infty}	v\mathcal{T}_n(u)(t,v) dv \le M_1(c_0) e^{(A-\kappa_n)t} + \kappa_n  M_1(c_0) \int_0^t e^{At - \kappa_n(t-s)} ds = M_1(c_0) e^{At}. \label{tau-2} 
\end{equation}

Therefore, \eqref{global nonnegative-2M}, \eqref{tau-1}, and \eqref{tau-2} imply that $P_n$ is an invariant set for the map $\mathcal{T}_n$.
	
Now we prove that $\mathcal{T}_n$ is a contraction on $P_n$  with respect to the distance $d_{P_n}$. Consider $(u_1,u_2)\in P_n^2$ and $t\in [0,t_n]$. From~\eqref{coagapriori2}, we obtain
\begin{align*}
	&\|\left(\mathcal{T}_n(u_1) -\mathcal{T}_n(u_2)\right)(t)\|_{L^1(0,\infty)} \\
	&\leq \int_{0}^{t} \left[ \|\left( Q_{n}(u_1)-Q_{n}(u_2)\right)(s)\|_{L^1(0,\infty)} + \kappa_n \| (u_1-u_2)(s)\|_{L^1(0,\infty)}\right] e^{-\kappa_n (t-s)} ds, \\
	& \leq \int_{0}^{t} \left[ 3 \beta_n \left( M_0(u_1(s)) + M_0(u_2(s))\right) + \kappa_n \right] ds\ d_{P_n}(u_1,u_2) \\
	&\leq t_{n} \left( \kappa_n  + 6\beta_{n} M_0(c_0) \right) d_{P_n}(u_1,u_2). 
\end{align*} 
Hence, the choice~\eqref{smallt_{0}} of $t_n$ leads us to
\begin{equation*}
	d_{P_n}\left( \mathcal{T}_n(u_1),\mathcal{T}_n(u_2) \right)\leq \frac{1}{2} d_{P_n}(u_1,u_2),
\end{equation*} 
which holds for all $(u_1,u_2) \in P_n^2$. Therefore, $\mathcal{T}_n$ is a strict contraction on $P_n$. Hence, by Banach contraction theorem, there is a unique fixed point $c_{n}\in P_n$ of $\mathcal{T}_n$ and $c_n$ is the sought-for mild solution to~\eqref{truncated-1}--\eqref{truncated-2} on $[0,t_n]$, as already pointed out. Moreover, from the definition of $P_n$, we have
\begin{equation} 
	0\leq M_0(c_{n}(t)) \leq M_0(c_0), \quad M_1(c_{n}(t))\leq M_1(c_0) e^{At} \label{finitenessofM_{1}},
\end{equation}
for $0\leq t\leq t_{n}$. 
	
Now, since $M_0(c_n(t_n))\le M_0(c_0)$ by~\eqref{finitenessofM_{1}}, we can repeat the same procedure as above with $c_{n}(t_{n})$ as initial condition and thereby extend $c_n$ uniquely on the time interval $[0,2t_{n}]$, showing as well that the extension of $c_n$ to $[t_{n},2t_{n}]$ also satisfies~\eqref{finitenessofM_{1}} for $[t_{n},2t_{n}]$. Iterating this argument provides the existence and uniqueness of a global mild solution $c_n$ to~\eqref{truncated-1}--\eqref{truncated-2} in the sense of \Cref{defin:mild}, which satisfies~\eqref{finitenessofM_{1}} for all $t\ge 0$. In particular, $c_n$ satisfies~\eqref{zeroth moment} and we are left with checking~\eqref{rate of first moment}.
	
Let $t\ge 0$ and $\eta_{R}\in C^{1}_c\left((0,\infty)\right)$ with
\begin{align*}
	\eta_{R}(v)=v \;\text{ for }\quad 0\leq v \leq R,\quad \eta_{R}(v)=0\; \text{ for }\;v\geq 2R 
\end{align*}
and 
\begin{align*}
	0\leq \eta_{R}(v)\leq 2 v, \quad |\eta^{\prime}_{R}(v)|\leq 2.
\end{align*}
Inserting $\eta_{R}$ in \eqref{truncated weak formulation}, we obtain
\begin{align}
	\int_{0}^{\infty}c_{n}(t,v)\eta_{R}(v) dv = \int_{0}^{\infty}\eta_{R}(v) c_{0}^{n}(v) dv +\int_{0}^{t} \int_{0}^{\infty} g(s,v) c_{n}(s,v) \eta_{R}^{\prime}(v) dvds \nonumber \\
	\quad + \int_{0}^{t} \int_{0}^{\infty} \int_0^\infty \widetilde{\eta_{R}}(v_{1},v_2) K_{n}(v_1,v_2) c_{n}(s,v_{1}) c_{n}(s,v_{2}) dv_{2}dv_{1}ds. \nonumber  
\end{align}
Since $c_{n}(t)$ and $c_{0}^{n}$ belong to $L_1^1(0,\infty)$ and 
\begin{equation*}
	\widetilde{\eta_{R}}(v_1,v_2)=0, \quad v_1+v_2\in (0,R), \qquad \big|\widetilde{\eta_{R}}(v_1,v_2)\big| \le 4 (v_1+v_2), \quad (v_1,v_2)\in (0,\infty)^2,
\end{equation*}
it follows from the boundedness of $K_n$ and~\eqref{growthassum} that we may take the limit $R\rightarrow \infty$ and deduce from Lebesgue's dominated convergence theorem that
\begin{equation*}
		\int_{0}^{\infty} v c_{n}(t,v) dv = \int_{0}^{\infty} v c_{0}^{n}(v) dv +\int_{0}^{t} \int_{0}^{\infty} g(s,v) c_{n}(s,v) dvds, 
\end{equation*}
hence~\eqref{rate of first moment}. Finally, $c_n$ is also a global weak solution to~\eqref{truncated-1}--\eqref{truncated-2} by \Cref{relation of weak and mild solution} and the proof of \Cref{fixed point theorem-2} is complete.
\end{proof}

Next, we investigate the behavior of the mild solution $c_{n}$ to~\eqref{truncated-1}--\eqref{truncated-2} for small sizes in the upcoming \Cref{-2 beta moment bound-0}. 

%%%%%%%%%%%%%%%%
\begin{lemma}\label{-2 beta moment bound-0} 
The mild solution $c_{n}$ to~\eqref{truncated-1}--\eqref{truncated-2} belongs to $L^\infty(0,\infty;L_{-2\beta}^1(0,\infty))$ and satisfies
\begin{equation}
	M_{-2\beta}(c_{n}(t))=\int_{0}^{\infty} v^{-2\beta} c_{n}(t,v) dv \leq 	M_{-2\beta}(c_{0}), \quad t\ge 0.\label{-2 beta moment bound}
\end{equation}
\end{lemma}
%%%%%%%%%%%%%%%%

\begin{proof}
	Let us define
		\begin{align}
		\phi_{\epsilon}(v):=(v+\epsilon)^{-2\beta} \;\text{for}\quad (\epsilon,v) \in (0,1)\times (0,\infty).\nonumber
	\end{align}
Then, it is clear that $\phi_{\epsilon}\in W^{1,\infty}(0,\infty)$.
	By \Cref{fixed point theorem-2}, we know that the mild solution is also a weak solution. Hence, applying $\phi_{\epsilon}$ in the weak formulation~\eqref{truncated weak formulation}, we have
	
	\begin{align} \label{weak moment beta}
		\int_{0}^{\infty} c_{n}(t,v)\phi_{\epsilon}(v)dv  =   	\int_{0}^{\infty} c_{0}^{n}(v)\phi_{\epsilon}(v)dv  
		+\int_{0}^{t}\int_{0}^{\infty} \phi_{\epsilon}^{\prime}(v)g(s,v)c_{n}(s,v)dvds \nonumber \\
		+\frac{1}{2}\int_{0}^{t}\int_{0}^{\infty} \tilde{\phi_{\epsilon}}(v_{1},v_{2})  K_{n}(v_{1},v_{2}) c_{n}(s,v_{1}) c_{n}(s,v_{2})dv_{2} dv_{1}ds, 
	\end{align}
	where
	\begin{align}
		\tilde{\phi_{\epsilon}}(v_{1},v_{2})=(v_{1}+v_{2}+\epsilon)^{-2\beta}-(v_{1}+\epsilon)^{-2\beta}-(v_{2}+\epsilon)^{-2\beta} \le 0. \label{z}
	\end{align}
	Let us first compute the first term in the right-hand side of \eqref{weak moment beta} 
	\begin{align}\label{moment beta I}
		\int_{0}^{\infty} c_{0}^{n}(v)\phi_{\epsilon}(v)dv   =\int_{0}^{\infty}(v+\epsilon)^{-2\beta}c_{0}^{n}(v) dv\leq \int_{0}^{\infty}v^{-2\beta}c_0(v) dv.
	\end{align}
	Then, we consider the second term in the right-hand side of \eqref{weak moment beta} 
	\begin{align}
		&\int_{0}^{t}\int_{0}^{\infty} \phi_{\epsilon}^{\prime}(v)g(s,v)c_{n}(s,v)dvds \nonumber\\
		=&-2\beta \int_{0}^{t}\int_{0}^{\infty}(v+\epsilon)^{-2\beta-1} g(s,v) c_{n}(s,v) dv ds. \nonumber
	\end{align}
	The non-negativity of $g$ and $c_{n}$ implies that
	\begin{align}\label{moment beta II}
		\int_{0}^{t}\int_{0}^{\infty} \phi_{\epsilon}^{\prime}(v)g(s,v)c_{n}(s,v)dvds\leq 0.
	\end{align}
	Finally, we evaluate the last term in the right-hand side of \eqref{weak moment beta} and infer from~\eqref{z} and the non-negativity of $K_n$ that
	\begin{align}
		\frac{1}{2}\int_{0}^{t}\int_{0}^{\infty} \tilde{\phi_{\epsilon}}(v_{1},v_{2})  K_{n}(v_{1},v_{2}) c_{n}(s,v_{1}) c_{n}(s,v_{2})dv_{2} dv_{1}ds \leq 0. \label{moment beta III}
	\end{align}
	Now, using~\eqref{moment beta I}, \eqref{moment beta II}, and~\eqref{moment beta III} in~\eqref{weak moment beta}, we have
	
	\begin{align}
		\int_{0}^{\infty} (v+\epsilon)^{-2\beta} c_{n}(t,v) dv \leq 	\int_{0}^{\infty} v^{-2\beta}c_0(v) dv. \nonumber
	\end{align}
	
	Taking the limit $\epsilon \rightarrow 0$ and using Fatou's lemma, we obtain the desired  result \eqref{-2 beta moment bound}.
\end{proof}

The following results are an immediate consequence of \Cref{fixed point theorem-2} and \Cref{-2 beta moment bound-0}.

%%%%%%%%%%%%%%%%
\begin{proposition}\label{Global prop-1}
	Let $n\geq 1$. There is a unique mild solution $c_n$ to \eqref{truncated-1}--\eqref{truncated-2} (which is also a weak solution according to \Cref{relation of weak and mild solution}) that satisfies, for $t>0$,
	\begin{align}
	\int_{0}^{\infty} c_n(t,v) dv&\leq 	\int_{0}^{\infty} c_0(v) dv, \nonumber \\
	\int_{0}^{\infty} vc_n(t,v) dv&= 	\int_{0}^{\infty} vc_{0}^{n}(v) dv +\int_{0}^{t} \int_{0}^{\infty} g(t,v)c_n(t,v) dv ds,\nonumber \\
	\int_{0}^{\infty} v^{-2\beta}c_n(t,v) dv&\leq 	\int_{0}^{\infty} v^{-2\beta}c_0(v) dv.\nonumber
	\end{align}
Moreover, for all $T>0$, there is a constant $C(T)>0$ such that 
\begin{equation*}
	\int_{0}^{\infty} vc_n(t,v) dv\leq C(T) \;\text{ for }\; t\in [0,T]. 
\end{equation*}
\end{proposition}
%%%%%%%%%%%%%%%%

We now focus on estimating superlinear moments, exploiting the condition~\eqref{coagassum} that the growth of $K$ is at most linear for large volumes.

%%%%%%%%%%%%%%%%
\begin{lemma} \label{Tail-1}
	 Let $\Gamma[c_0]$ be the set of non-negative convex functions
	$j\in \mathcal{C}^1\left([0,\infty)\right)$ such that $j^{\prime}$ is a concave function on  $[0,\infty)$ with $ j(0)=0$, $j^{\prime}(0)\geq 0 $ and
	\begin{align}\label{N}
		N(j,c_0) := \int_{0}^{\infty} j(v)c_0(v) dv < \infty.
	\end{align}
	Let $c_n$ be the global mild solution to~\eqref{truncated-1}--\eqref{truncated-2} and $T\in (0,\infty)$. Then there is a constant $C(T)$ depending only on $A$, $k$, $c_0$, $j$, and $T$ such that
	\begin{align}
		\int_{0}^{\infty} j(v)c_n(t,v) dv \leq C(T) \text{ for }\; t\in (0,T). \label{j0}
	\end{align}
\end{lemma}
%%%%%%%%%%%%%%%%

\begin{proof}
	We set $C(T)$ as any positive constant which depends only on $A$, $K$, $j$, and $T$ throughout the proof. For $R\geq 1$, set $j_R(v):=\min(j(v),j(R))$,  $v\in (0,\infty).$ Clearly, $j_R\in W^{1,\infty}(0,\infty)$ and, from \cite[Lemmas~A.1--A.2]{Laurencot2001}, we have, for $(v_{1},v_{2})\in (0,\infty)^2$, 
	\begin{align}
		(v_{1}+v_{2})\left(j_R(v_{1}+v_{2})-j_R(v_{1})-j_R(v_{2})\right)&\leq 2\left(v_{2}j_R(v_{1})+v_{1}j_R(v_{2})\right),  \label{convex inequality-1}\\
		0\leq v_{1}j_{R}^{\prime}(v_{1})&\leq 2j_{R}(v_{1}). \label{convex inequality-2}
	\end{align}
 Let $t\in [0,T]$. We infer from~\eqref{truncated weak formulation} with $\phi=j_R$ that
	\begin{align}
		\int_{0}^{\infty} j_R(v)c_n(t,v) dv &=\int_{0}^{\infty} j_R(v)c_{0}^{n}(v) dv + \int_{0}^{t}\int_{0}^{\infty} j^{\prime}_R(v)g(s,v)c_n(s,v) dvds \nonumber \\
		&\; +\frac{1}{2}\int_{0}^{t}\int_{0}^{\infty} \int_{0}^{\infty} \widetilde{j_R}(v_{1},v_2) K_n(v_{1},v_{2})c_n(s,v_{1}) c_n(s,v_{2}) dv_2dv_{1}ds, \nonumber 
	\end{align}
where $\widetilde{j_R}(v_{1}v_2):=j_R(v_{1}+v_{2})-j_R(v_{1})-j_R(v_{2})$ for $(v_{1},v_{2})\in (0,\infty)^2$. Using \eqref{N}, we obtain
  \begin{align}
		\int_{0}^{\infty} j_R(v)c_n(t,v) dv	&\leq N(j,c_0) + \int_{0}^{t}\int_{0}^{\infty} j^{\prime}_R(v)g(s,v)c(s,v) dvds \nonumber \\
		\; &+\frac{1}{2}\int_{0}^{t}\int_{0}^{\infty} \int_{0}^{\infty} \widetilde{j_R}(v_{1},v_2) K_n(v_{1},v_{2})c_n(s,v_{1}) c_n(s,v_{2}) dv_2dv_{1}ds. \label{j1}
	\end{align}
 Using~\eqref{characteristic-1} and~\eqref{convex inequality-2} in the second term of the right-hand side of \eqref{j1}, we get
	\begin{align}
		\int_{0}^{t}\int_{0}^{\infty} j^{\prime}_R(v)g(s,v)c_n(s,v) dvds &\leq A \int_{0}^{t}\int_{0}^{\infty} j^{\prime}_R(v)vc_n(s,v) dvds \nonumber \\
		&\leq 2A \int_{0}^{t}\int_{0}^{\infty} j_R(v)c_n(s,v) dvds. \label{j2}
	\end{align}
Let us now estimate the last term in the right-hand side of \eqref{j1}. For that purpose, we divide the above integral in the following manner:
\begin{align}\label{tail-1}
	\int_{0}^{\infty}\int_{0}^{\infty} \widetilde{j_R}(v_{1},v_2) K_n(v_{1},v_{2})c_n(s,v_{1})c_n(s,v_{2}) dv_{1}dv_{2} = \sum_{i=1}^{4}Z_{i}, 
\end{align}
where
\begin{align}
	Z_{1} :=\int_{0}^{1}\int_{0}^{1} \widetilde{j_R}(v_{1},v_2) K_n(v_{1},v_{2})c_n(s,v_{1})c_n(s,v_{2}) dv_{1}dv_{2},  \nonumber
\end{align}
\begin{align}
	Z_{2} :=\int_{0}^{1}\int_{1}^{\infty} \widetilde{j_R}(v_{1},v_2) K_n(v_{1},v_{2})c_n(s,v_{1})c_n(s,v_{2}) dv_{1}dv_{2},  \nonumber
\end{align}
\begin{align}
	Z_{3} :=\int_{1}^{\infty}\int_{0}^{1} \widetilde{j_R}(v_{1},v_2) K_n(v_{1},v_{2})c_n(s,v_{1})c_n(s,v_{2}) dv_{1}dv_{2} =Z_2,  \nonumber
\end{align}
\begin{align}
	Z_{4} :=\int_{1}^{\infty}\int_{1}^{\infty} \widetilde{j_R}(v_{1},v_2) K_n(v_{1},v_{2})c_n(s,v_{1})c_n(s,v_{2}) dv_{1}dv_{2}.  \nonumber
\end{align}
Applying~\eqref{coagassum}, \eqref{convex inequality-1} and using the definition of $K_n$ and \Cref{Global prop-1}, the term $Z_{1}$ can be evaluated as  
\begin{align*}
	Z_{1}&\leq 2\int_{0}^{1}\int_{0}^{1}\left( \frac{v_{2}j_R(v_{1})+v_{1}j_R(v_{2})}{v_{1}+v_{2}}\right) K_n(v_{1},v_{2})c_n(s,v_{1})c_n(s,v_{2}) dv_{1}dv_{2}  \\ 
	&\leq 2k\int_{0}^{1}\int_{0}^{1}\left( j_R(v_{1})+j_R(v_{2})\right) (v_{1}v_{2})^{-\beta}c_n(s,v_{1})c_n(s,v_{2}) dv_{1}dv_{2}  \\ 
	&\leq 4k j(1)\int_{0}^{1}\int_{0}^{1} v_{1}^{-\beta} v_{2}^{-\beta}c_n(s,v_{1})c_n(s,v_{2}) dv_{1}dv_{2}   \\
	&\leq 4k j(1) M^2_{-2\beta}(c_0).   	
\end{align*}
Analogously, one can estimate
\begin{align*}
	Z_{2} =Z_3 & \leq 2k \int_0^1 \int_1^\infty \left( \frac{v_1 v_2^{1-\beta}}{v_1+v_2} j_R(v_1) + \frac{v_1^2 v_2^{-\beta}}{v_1+v_2} j_R(v_2) \right) c_n(s,v_1) c_n(s,v_2) dv_1dv_2\\ 
	& \leq 2k \int_0^1 \int_1^\infty \left( v_2^{1-\beta} j_R(v_1) + v_1 v_2^{-\beta} j(1) \right) c_n(s,v_1) c_n(s,v_2) dv_1dv_2\\ 
	& \leq 2k \left[  M_{-2\beta}(c_0) \int_{0}^{\infty} j_R(v_{1})c_n(s,v_{1})dv_{1} + j(1) M_{-2\beta}(c_0) M_1(c_0) e^{AT} \right], 
\end{align*}
and
\begin{align*}
	Z_{4} & \le 2k \int_1^\infty \int_1^\infty \big( v_1 j_R(v_2) + v_2 j_R(v_1) \big) c_n(s,v_1) c_n(s,v_2) dv_1 dv_2 \\
	& \leq 4k M_1(c_0) e^{AT} \int_{0}^{\infty} j_R(v_{1})c_n(s,v_{1})dv_{1}  
	\leq  C(T)  \int_{0}^{\infty} j_R(v_{1})c_n(s,v_{1})dv_{1}.  	
\end{align*}
Gathering the estimates on $Z_{1}$, $Z_{2}$, $Z_{3}$ and $Z_{4}$ and inserting them in~\eqref{tail-1}, we obtain
\begin{align}
	&\frac{1}{2}\int_{0}^{\infty}\int_{0}^{\infty} \widetilde{j_R}(v_{1},v_2) K_n(v_{1},v_{2})c_n(s,v_{1})c_n(s,v_{2}) dv_{1}dv_{2} \nonumber \\
	&\leq C(T) + C(T) \int_{0}^{\infty} j_R(v_{1})c_n(s,v_{1})dv_{1}. \label{j3} 
\end{align}
Now, using~\eqref{j2} and~\eqref{j3} in~\eqref{j1}, we obtain
	\begin{equation*}
		\int_{0}^{\infty} j_R(v)c_n(t,v)dv  \leq C(T)+ C(T)\int_{0}^{t}  \int_{0}^{\infty} j_R(v)c_n(s,v)dvds.
	\end{equation*}
	Hence, by Gronwall's inequality, we have 
	\begin{equation*}
		\int_{0}^{\infty} j_R(v)c_n(t,v)dv  \leq C(T),
	\end{equation*}
and~\eqref{j0} follows from the above inequality by Fatou's Lemma after letting $R\rightarrow \infty$.
\end{proof}

%%%%%%%%%%%%%%%%
\begin{corollary}\label{second moment lemma-1}
	Suppose that $c_0$ belongs to $L_2^1(0,\infty)$. Then, for every $T>0$,
	there exists a positive constant $C_{sm}(T)$ such that
	\begin{equation*}%\label{secondmomentbound-0}
	M_{2}(c_{n}(t)) \leq C_{sm}(T), \quad t \in [0, T], \quad n\geq 1.
	\end{equation*}
\end{corollary}
%%%%%%%%%%%%%%%%

\begin{proof}
	 The function $J_{2} : v\mapsto  v^2$ is convex on $[0,\infty)$ with concave derivative, and satisfies $J_{2}(0)=J_{2}^{\prime}(0) = 0$. Hence, $J_2\in \Gamma[c_0]$ and \Cref{second moment lemma-1} straightforwardly follows from \Cref{Tail-1} with  $j=J_{2}$.
\end{proof}

We are left with the tail behavior of $c_{n}$, which we analyze in the next lemma.

%%%%%%%%%%%%%%%%
\begin{lemma} \label{Tail-2}
	For $T>0$,
	\begin{equation} \label{first moment uniform integrability}
		\lim_{R\to\infty} \sup_{n\geq 1} \sup_{t\in [0,T]}	\int_{R}^{\infty} v c_n(t,v)dv  = 0.  
	\end{equation}	
\end{lemma}
%%%%%%%%%%%%%%%%

\begin{proof}
A refined form of the de la Vall\'{e}e-Poussin theorem, originally derived in \cite[Proposition~I.1.1]{Le1977} and also reported in \cite[Theorem~7.1.6]{BLL_book}, establishes the existence of a function $j_0\in \Gamma[c_0]$ satisfying also
\begin{equation}\label{convexinfty}
	\lim_{v \rightarrow \infty} j_{0}^{\prime}(v) = \lim_{v \rightarrow \infty} \frac{j_0(v)}{v} = \infty.
\end{equation}
For $R\geq 1,$ we have
	\begin{align}
		&\sup_{n\geq 1} \sup_{t\in [0,T]}	\int_{R}^{\infty} vc_n(t,v)dv  \nonumber \\
		&\leq \sup_{n\geq 1} \sup_{t\in [0,T]} \sup_{v\in [R,\infty)}\left(\frac{v}{j_0(v)}\right)	\int_{R}^{\infty} j_0(v)c_n(t,v)dv. \nonumber
	\end{align}
Since $j_0\in \Gamma[c_0]$, we can utilize \Cref{Tail-1} to obtain
	\begin{align}	
		\sup_{n\geq 1} \sup_{t\in [0,T]}	\int_{R}^{\infty} vc_n(t,v)dv \leq C(T) \sup_{v\in [R,\infty)}\left(\frac{v}{j_0(v)}\right).  \nonumber
	\end{align}	
This inequality implies~\eqref{first moment uniform integrability} by~\eqref{convexinfty}.
\end{proof}

We now examine the behavior of $c_{n}$ on small measurable subsets of $(0,\infty)$. Accordingly, we present two lemmas; detailed proofs can be found in \cite[Lemmas~4.3--4.4]{Laurencot2001}.

%%%%%%%%%%%%%%%%
\begin{lemma}\label{Auxiliary lemma-1}
	Let $Z \in \mathcal{M}$ with $|Z|<\infty$ and m be a positive integer. Then there exist $2^m$ subsets $(Z_l) \in \mathcal{M}$ such that 
	\begin{align*}
		|Z_l|\leq |Z|, \;  l\in \left\{1,.....2^m\right\}, \\
		2^mZ:=\left\{2^mv, v\in Z\right\}\subset\bigcup_{l=1}^{2^m} \; Z_l ,
	\end{align*}
	 where $\mathcal{M}$ is the set of all measurable subsets of $(0,\infty)$ and $|Z|$ denotes the Lebesgue measure of $Z$.
\end{lemma}
%%%%%%%%%%%%%%%%

%%%%%%%%%%%%%%%%
\begin{lemma} \label{Auxiliary lemma-2}
Let $T>0$ and $Z \in \mathcal{M}$ with $|Z|<\infty$. Then there exists an integer $m_0$  depending only on $A$ and $T$ such that
	\begin{equation*}%\label{auxiliary inequality}
		|Y(s;t,Z)|\leq 2^{m_0}|Z|, \quad (s,t)\in [0,T]^2 .
	\end{equation*}
\end{lemma}
%%%%%%%%%%%%%%%%

Having completed this preparation, we are poised to articulate the main result of this section.

%%%%%%%%%%%%%%%%
\begin{proposition} \label{uniform integrability proposition}
	Given $T>0$ and $\epsilon > 0$, there exists $\delta_\epsilon > 0$ such that for all $n\ge1$ and  $t\in [0,T]$
		\begin{align}
			\int_{Z} v^{-\beta} c_n(t,v) dv \leq \epsilon  \hspace{1cm} \text{whenever} \hspace{1cm} |Z| \leq \delta_\epsilon.\nonumber
		\end{align}
\end{proposition}
%%%%%%%%%%%%%%%%

\begin{proof}
We take $C_p=2^{m_0}$ where $m_0$ is given in \Cref{Auxiliary lemma-2}  .\\
Let $t\in [0,T], n\geq 1, \;\delta \in (0,1)$ and define
\begin{equation*}
	\mathcal{E}_{n,\delta}(t):=	\sup  \left\{\!\begin{aligned}	\int_{0}^{\infty} \chi _{Z} (v) v^{-\beta}c_n(t,v)dv, Z \in \mathcal{M} \;\text{with} \; |Z|\leq \delta \end{aligned}\right\}.  \nonumber
\end{equation*}
Consider $Z \in \mathcal{M}$ having $|Z|\leq \delta$. According to  \Cref{Auxiliary lemma-1} and \Cref{Auxiliary lemma-2}, for each $t\in [0,T]$ and $s\in[0,t]$, we have $C_p$ subsets $(Z^{s,t}_l)\in \mathcal{M}$ such that $|Z^{s,t}_l|\leq \delta$ and 
\begin{equation}
	Y(s;t,Z) \subset \bigcup_{l=1}^{C_p} Z^{s,t}_l. \label{C}
\end{equation}
Next, since $c_n$ is a mild solution to~\eqref{truncated-1}--\eqref{truncated-2}, we obtain
\begin{align} 
	\int_{Z} v^{-\beta} c_n(t,v) dv	= & \int_{0}^{t}	\int_{Z} v_{1}^{-\beta} Q_n(c_n) (s,Y(s;t,v_{1})) J(s;t,v_{1})dv_{1}ds  \nonumber  \\ 
	+	& \int_{Z} v^{-\beta} c_0(Y(0;t,v))J(0;t,v)dv, \label{uniform integrability1} 
\end{align}
Now let
\begin{equation*}
	I_1:=\int_{0}^{t}\int_{Z} v_{1}^{-\beta} Q_n(c_n) (s,Y(s;t,v_{1})) J(s;t,v_{1})dv_{1}ds.  
\end{equation*}
Therefore, by \eqref{characteristic-4}, and the non-negativity of $c_n$ and $K_{n}$, we have 
\begin{equation*}
	 I_1 \leq \frac{1}{2} \int_{0}^{t}	\int_{Y(s;t,Z)} Y^{-\beta}(t;s,v_{1}) \int_{0}^{v_{1}} K_n(v_{1}-v_{2},v_{2})c_n(s,v_{1}-v_{2})c_n(s,v_{2}) dv_{2}dv_{1}ds. 
\end{equation*}
Owing to \eqref{inequality-13}, we have
\begin{equation*}
I_1 \leq \frac{1}{2} \int_{0}^{t}	\int_{Y(s;t,Z)} v_{1}^{-\beta} \int_{0}^{v_{1}} K_n(v_{1}-v_{2},v_{2})c_n(s,v_{1}-v_{2})c_n(s,v_{2})dv_{2}dv_{1}ds.  
\end{equation*}
Hence, by Fubini's theorem,
\begin{align}
	I_1	&\leq \frac{1}{2} \int_{0}^{t}	\int_{0}^{\infty} \int_{0}^{\infty}  (v_{2}+v_{1})^{-\beta} \chi_{Y(s;t,Z)}(v_{2}+v_{1})  K_n(v_{1},v_{2})c_n(s,v_{1})c_n(s,v_{2}) dv_{1}dv_{2}ds = \sum_{i=1}^{4}W_{i}, \nonumber 
\end{align}
where
\begin{equation*}
	W_1 := \int_{0}^{t}	\int_{0}^{1} \int_{0}^{1}  (v_{1}+v_{2})^{-\beta} \chi_{Y(s;t,Z)}(v_{1}+v_{2})  K_n(v_{1},v_{2})c_n(s,v_{1})c_n(s,v_{2})dv_{2}dv_{1}ds,  
\end{equation*}
\begin{equation*}
	W_2 :=\int_{0}^{t}	\int_{0}^{1} \int_{1}^{\infty}  (v_{1}+v_{2})^{-\beta} \chi_{Y(s;t,Z)}(v_{1}+v_{2})  K_n(v_{1},v_{2})c_n(s,v_{1})c_n(s,v_{2})dv_{2}dv_{1}ds,  
\end{equation*}
\begin{equation*}
	W_3 :=\int_{0}^{t}	\int_{1}^{\infty} \int_{0}^{1}  (v_{1}+v_{2})^{-\beta} \chi_{Y(s;t,Z)}(v_{1}+v_{2})  K_n(v_{1},v_{2})c_n(s,v_{1})c_n(s,v_{2})dv_{2}dv_{1}ds = W_2,
\end{equation*}
and
\begin{equation*}
	W_4 :=\int_{0}^{t}	\int_{1}^{\infty} \int_{1}^{\infty}  (v_{1}+v_{2})^{-\beta} \chi_{Y(s;t,Z)}(v_{1}+v_{2})  K_n(v_{1},v_{2})c_n(s,v_{1})c_n(s,v_{2})dv_{2}dv_{1}ds.  
\end{equation*}
By the definition of $K_{n}$ and \eqref{coagassum},
\begin{align*}
	W_1 &\leq k\int_{0}^{t}	\int_{0}^{1} \int_{0}^{1}  v_{1}^{-\beta} \chi_{Y(s;t,Z)}(v_{1}+v_{2}) v_{1}^{-\beta}v_{2}^{-\beta} c_n(s,v_{1})c_n(s,v_{2}) dv_{2}dv_{1}ds  \\
	&\leq k\int_{0}^{t}	\int_{0}^{1}  v_{1}^{-2\beta}\int_{0}^{1}  \chi_{-v_{1}+Y(s;t,Z)}(v_{2}) v_{2}^{-\beta}c_n(s,v_{1})c_n(s,v_{2}) dv_{2}dv_{1}ds.
\end{align*}
Next, \Cref{Global prop-1}, \eqref{C}, and the translation invariance of Lebesgue measure give 
\begin{align*}
	W_1 &\leq k\sum_{l=1}^{C_p}\int_{0}^{t}	\int_{0}^{1}  v_{1}^{-2\beta}c_n(s,v_{1})\left(\int_{0}^{1}  \chi_{-v_{1}+Z^{s,t}_l}(v_{2})  v_{2}^{-\beta}c_n(s,v_{2})dv_{2}\right)dv_{1}ds \\
	&\leq k\sum_{l=1}^{C_p}\int_{0}^{t} \mathcal{E}_{n,\delta}(s) \int_{0}^{1}  v_{1}^{-2\beta}c_n(s,v_{1})dv_{1}ds \leq  C(T) \int_{0}^{t} \mathcal{E}_{n,\delta}(s)ds.\nonumber 
\end{align*}
Similarly,
\begin{align*}
	W_2=W_3 &\leq k\int_{0}^{t}	\int_{0}^{1} \int_{1}^\infty \chi_{Y(s;t,Z)}(v_{1}+v_{2}) v_{1}^{-\beta}v_{2} c_n(s,v_{1})c_n(s,v_{2}) dv_{2}dv_{1}ds \\
	&\leq k\int_{0}^{t}	\int_{1}^{\infty} v_{2} c_n(s,v_2) \int_{0}^{1} v_1^{-\beta} \chi_{-v_{2}+Y(s;t,Z)}(v_{1}) c_n(s,v_{1}) dv_{1}dv_{2}ds \\
	&\leq  C(T) \int_{0}^{t} \mathcal{E}_{n,\delta}(s)ds,
\end{align*}
and
\begin{align*}
	W_4 &\le k\int_{0}^{t}	\int_{1}^\infty \int_{1}^\infty \big( v_{1} + v_{2} \big)^{1-\beta} \chi_{Y(s;t,Z)}(v_{1}+v_{2}) c_n(s,v_{1})c_n(s,v_{2}) dv_{2}dv_{1}ds \\
	& \le k\int_{0}^{t}	\int_{1}^\infty \int_{1}^\infty \big( v_{1} v_{2}^{-\beta} + v_{1}^{-\beta} v_{2} \big) \chi_{Y(s;t,Z)}(v_{1}+v_{2}) c_n(s,v_{1})c_n(s,v_{2}) dv_{2}dv_{1}ds \\
	& \le 2k \int_{0}^{t}	\int_{1}^\infty v_1 c_n(s,v_1) \int_{1}^\infty v_{2}^{-\beta} \chi_{-v_1+Y(s;t,Z)}(v_{2}) c_n(s,v_{1}) dv_{2}dv_{1}ds \\
	&\leq  C(T) \int_{0}^{t} \mathcal{E}_{n,\delta}(s)ds. 
\end{align*}
From the estimates of $W_1$, $W_2$, $W_3$, and $W_4$, we infer that
\begin{equation}
	I_1\leq  C(T) \int_{0}^{t} \mathcal{E}_{n,\delta}(s)ds. \label{uniform integrability2}
\end{equation}

Next, by~\eqref{characteristic-4} and~\eqref{inequality-13},  
\begin{align*}
\int_{Z} v^{-\beta} c_0(Y(0;t,v))J(0;t,v)dv  &= \int_{Y(0;t,Z)} Y^{-\beta}(t;0,v) c_0(v)dv \\
&\leq \int_{Y(0;t,Z)} v^{-\beta } c_0(v)dv,
\end{align*}
and it follows from~\eqref{C} that
\begin{equation}
	\int_{Z} v^{-\beta} c_0(Y(0;t,v))J(0;t,v)dv \leq \sum_{l=1}^{C_p} \int_{Z_l^{0,t}} v^{-\beta } c_0(v)dv \leq C_p\mathcal{E}_{\delta}(0), \label{uniform integrability3} 
\end{equation}
where
\begin{equation*}
	\mathcal{E}_{\delta}(0):=	\sup  \left\{\!\begin{aligned}	\int_{0}^{\infty} \chi_{Z} (v) v^{-\beta}c_{0}(v)dv, Z \in \mathcal{M} \;\text{with} \; |Z|\leq \delta \end{aligned}\right\}.  \nonumber
\end{equation*}
Using~\eqref{uniform integrability2} and~\eqref{uniform integrability3} in~\eqref{uniform integrability1}, we get
\begin{align}
	\int_{Z} v^{-\beta} c^n(t,v) dv	\leq C_p \mathcal{E}_{\delta}(0)+ C(T) \int_{0}^{t} \mathcal{E}_{n,\delta}(s)ds. \nonumber
\end{align}
Taking the supremum on both sides over the  set $Z\in \mathcal{M}$ with $|Z|\leq \delta$, we have
\begin{equation*}
\mathcal{E}_{n,\delta}(t) \leq C_p \mathcal{E}_{\delta}(0)+ C(T) \int_{0}^{t} \mathcal{E}_{n,\delta}(s)ds, \quad t\in [0,T]. 
\end{equation*}
By Gronwall's inequality, we get
\begin{equation}
	\mathcal{E}_{n,\delta}(t) \leq C(T) \mathcal{E}_{\delta}(0), \quad t\in [0,T]. \label{del1}  
\end{equation}
Now, by the absolute continuity of the Lebesgue integral, for a given $\epsilon > 0$, we can find $\delta_\epsilon >0$ such that
\begin{equation}
	\mathcal{E}_{\delta_\epsilon}(0) \leq \frac{\epsilon}{ C(T)}. \label{del2}
\end{equation}
From~\eqref{del1} and~\eqref{del2}, we get
\begin{equation*}
	\mathcal{E}_{n,\delta_\epsilon}(t) \leq \epsilon, \quad t\in [0,T], 
\end{equation*}
for all $n\ge 1$, and thus complete the proof of \Cref{uniform integrability proposition}.
\end{proof}

 From \Cref{uniform integrability proposition}, \Cref{Tail-2}, and the Dunford-Pettis theorem, we infer that
\begin{align}
	\begin{cases}~&{ (c_n(t))_{n\geq 1} \text{~~~is ~~weakly ~~compact ~~in}} \\
		&{ ~~ L^1(0,\infty)~~\text{ for~~ each }~~t \in [0,T].} \end{cases}	\label{weakly compact}	
\end{align}

%%%%%%%%%%%%%%%%
%%%%%%%%%%%%%%%%
%\subsection {Equicontinuity in Time}
%%%%%%%%%%%%%%%%
%%%%%%%%%%%%%%%%

We now proceed to prove that $(c_n)_{n\geq 1}$ is weakly equicontinuous with respect to $t$ in $ L^1(0,\infty)$.

%%%%%%%%%%%%%%%%
\begin{lemma} \label{time equicontinuity}
	For every $T>0$ and $\phi \in L^{\infty}(0,\infty)$, the following result is true.
\begin{equation}
	\lim_{h\rightarrow 0} \sup_{t\in [0,T-h]}	\sup_{n\geq 1}	\left|\int_{0}^{\infty}  (c_n(t+h,v)-c_n(t,v))\phi(v)dv \right| =0.\label{equicontinuity-0}
\end{equation}
\end{lemma}
%%%%%%%%%%%%%%%%

\begin{proof}
	Let $T>0$, $\phi \in  C_{c}^{1}(0,\infty)$, $n\geq 1$, $h\in (0,T)$, and $t\in (0,T-h)$. By~\eqref{truncated weak formulation}, we have
\begin{align}\label{EQN-1}
	&\left|\int_{0}^{\infty}(c_n(t+h,v)-c_n(t,v))\phi(v)dv\right|\nonumber\\
	& \quad	 \leq \|\phi^{\prime}\|_{L^{\infty}(0,\infty)} \int_{t}^{t+h} \int_{0}^{\infty} g(t,v) c_n(s,v) dvds\nonumber  \\ 
	& \qquad + \frac{1}{2} \int_{t}^{t+h} \int_{0}^{\infty} \int_{0}^{\infty} \left|\tilde{\phi}(v_{1},v_{2})\right| K_n(v_{1},v_{2}) c_n(s,v_{1}) c_n(s,v_{2}) dv_{1}dv_{2}ds,   
\end{align}
with ~~$\tilde{\phi}(v_{1},v_{2})=\phi(v_{1}+v_{2})-\phi(v_{1})-\phi(v_{2})$.

First, by~\eqref{characteristic-1} and \Cref{Global prop-1}, we get
\begin{align}
	\int_{t}^{t+h} \int_{0}^{\infty} g(s,v) c_n(s,v)dvds\leq C(T)h. \label{EQN-2}
\end{align}
Let us now estimate the second term on the right-hand side of \eqref{EQN-1}:
\begin{align} 
	&\frac{1}{2} \int_{t}^{t+h} \int_{0}^{\infty} \int_{0}^{\infty} \left|\tilde{\phi}(v_{1},v_{2})\right| K_n(v_{1},v_{2}) c_n(s,v_{1}) c_n(s,v_{2}) dv_{1}dv_{2}ds \nonumber\\  
	&\leq \frac{3}{2} \|\phi\|_{L^{\infty}(0,\infty)} \int_{t}^{t+h} \int_{0}^{\infty} \int_{0}^{\infty} K_n(v_{1},v_{2}) c_n(s,v_{1}) c_n(s,v_{2}) dv_{1}dv_{2}ds. \label{equicontinuity1}
\end{align}
In order to further estimate the inequality \eqref{equicontinuity1}, let us consider
\begin{equation*}
	\int_{0}^{\infty}\int_{0}^{\infty}K_n(v_{1},v_{2})c_n(s,v_{1})c_n(s,v_{2}) dv_{1}dv_{2}ds= \sum_{i=1}^{4} \widehat{W}_{i} , 
\end{equation*}
where
\begin{equation*}
	\widehat{W}_1 := \int_{0}^{1}\int_{0}^{1}K_n(v_{1},v_{2})c_n(s,v_{1})c_n(s,v_{2}) dv_{1}dv_{2},   
\end{equation*}
\begin{equation*}
	\widehat{W}_2 :=\int_{0}^{1}\int_{1}^{\infty}K_n(v_{1},v_{2})c_n(s,v_{1})c_n(s,v_{2}) dv_{1}dv_{2},   
\end{equation*}
\begin{equation*}
	\widehat{W}_3 :=\int_{1}^{\infty}\int_{0}^{1}K_n(v_{1},v_{2})c_n(s,v_{1})c_n(s,v_{2}) dv_{1}dv_{2} = \widehat{W}_2, 
\end{equation*}
and
\begin{equation*}
	\widehat{W}_4 :=\int_{1}^{\infty}\int_{1}^{\infty}K_n(v_{1},v_{2})c_n(s,v_{1})c_n(s,v_{2}) dv_{1}dv_{2}.
\end{equation*}
With the help of definition of $K_{n}$, \eqref{coagassum}  and \Cref{Global prop-1}, $\widehat{W}_{1}$ can be estimated as
\begin{align*}
	\widehat{W}_{1} &\leq k\int_{0}^{1}\int_{0}^{1}v_{1}^{-\beta}v_{2}^{-\beta}c_n(s,v_{1})c_n(s,v_{2}) dv_{1}dv_{2} \\
	&\leq k \left(\int_{0}^{\infty}v_{1}^{-2\beta}c_n(s,v_{1})dv_{1}\right) \left(\int_{0}^{\infty}v_{2}^{-2\beta}c_n(s,v_{2}) dv_{2}\right) \leq k M_{-2\beta}^{2}(c_{0}). \nonumber
\end{align*}
Similarly
\begin{equation*}
\widehat{W}_{2} =\widehat{W}_3 \leq k\left(\int_{0}^{\infty}v_{2}^{-2\beta}c_n(s,v_{2})dv_{2}\right)\left(\int_{0}^{\infty}v_{1}c_n(s,v_{1}) dv_{1}\right)\leq k e^{AT}M_{-2\beta}(c_{0}) M_{1}(c_{0}), \nonumber
\end{equation*}
and
\begin{align*}
		\widehat{W}_{4}=&\int_{1}^{\infty}\int_{1}^{\infty}K_n(v_{1},v_{2})c_n(s,v_{1})c_n(s,v_{2}) dv_{1}dv_{2} \\
	&\leq 2k \left(\int_{0}^{\infty}v_{1}c_n(s,v_{1})dv_{1}\right)\left(\int_{0}^{\infty}c_n(s,v_{2}) dv_{2}\right)\leq 2ke^{AT}M_{1}(c_{0})M_{0}(c_{0}). 
\end{align*}
From the estimates on $\left\{ \widehat{W}_{i}\right\}_{1\le i\le 4}$, we infer that
\begin{equation}\label{equicontinuity2}
	\int_{0}^{\infty}\int_{0}^{\infty}K_n(v_{1},v_{2})c_n(s,v_{1})c_n(s,v_{2}) dv_{1}dv_{2} \leq C(T), \quad s\in [0,T].
\end{equation}
Inserting~\eqref{equicontinuity2} into~\eqref{equicontinuity1}, we obtain
\begin{align}
	&\frac{1}{2} \int_{t}^{t+h} \int_{0}^{\infty} \int_{0}^{\infty} \left|\tilde{\phi}(v_{1},v_{2})\right| K_n(v_{1},v_{2}) c_n(s,v_{1}) c_n(s,v_{2}) dv_{1}dv_{2}ds \leq  C(T) \|\phi\|_{L^{\infty}(0,\infty)}h. \label{equicontinuity3}
\end{align}
Gathering~\eqref{EQN-1}, \eqref{EQN-2} and~\eqref{equicontinuity3}, we end up with 
\begin{equation*}
	\left|\int_{0}^{\infty}  (c_n(t+h,v)-c_n(t,v))\phi(v)dv\right|  \leq C(T) \|\phi\|_{W^{1,\infty}(0,\infty)}h
\end{equation*}
for all $h\in (0,T)$, $t\in [0,T-h]$, and $n\ge 1$. Therefore, \eqref{equicontinuity-0} is true for every  $\phi \in  C^{1}_c(0,\infty)$. We next use a density argument as in the proof of \cite[Lemma~4.5]{Laurencot2001} to extend the validity of~\eqref{equicontinuity-0} to arbitrary functions in $L^\infty(0,\infty)$.
\end{proof}

Thus, from \Cref{time equicontinuity}, we infer  that
\begin{align}
	\begin{cases}~&{ (c_n)_{n\geq 1} \text{~~~is ~~weakly ~~ equicontinuous~~in } L^1(0,\infty)} \\
		&{ \text{ for~~ each }~~t \in [0,T]}\;  ,\end{cases}\label{weakly equicontinuous}		
\end{align}
see \cite[Definition~1.3.1]{Vrabie2003}. Due to~\eqref{weakly compact} and~\eqref{weakly equicontinuous}, we may apply a variant of the Arzel\'{a}-Ascoli theorem  \cite[Theorem~1.3.2]{Vrabie2003} to deduce that the sequence $(c_n)$ is relatively compact in $C([0,T];w-L^1(0,\infty))$ for any $T>0$. Therefore, by a diagonal argument, we conclude that there exists a subsequence $(c_n)$ (not relabeled) which converges to some limit function $c$ in  $C\left([0,\infty);w-L^1(0,\infty)\right)$ in the following sense:
\begin{align}\label{convergence-3}
\lim_{n\rightarrow \infty} \sup_{t\in [0,T]}\left| \int_{0}^{\infty}(c_n(t,v)-c(t,v))\phi(v) dv\right|=0
\end{align}
for each $T\in (0,\infty)$ and $\phi \in L^{\infty}(0,\infty)$.

We shall improve this convergence in the next lemma but first combine \Cref{Global prop-1} and the convergence~\eqref{convergence-3} to derive additional properties on $c$. First, $c$ is clearly non-negative due to~\eqref{convergence-3} and the non-negativity of $c_n$ for all $n\ge 1$. Next, for $m\in [-2\beta,1]$, we infer from \Cref{Global prop-1} and~\eqref{convergence-3} that, for $T>0$, $t\in [0,T]$, and $R>1$,
\begin{equation*}
	\int_{\frac{1}{R}}^{R} v^m c(t,v) dv = \lim_{n\to\infty} \int_{\frac{1}{R}}^{R} v^m c(t,v) dv \le C(T). 
\end{equation*}
We let $R\to \infty$ and use Fatou's lemma to conclude that
\begin{equation}
	\sup_{t\in [0,T]} M_m(c(t)) \le C(T), \qquad m\in [-2\beta,1]. \label{M}
\end{equation}
Similarly, we infer from \Cref{Tail-2} and~\eqref{convergence-3} that
\begin{equation}
	\lim_{R\to\infty} \sup_{t\in [0,T]} \int_R^\infty v c(t,v) dv = 0. \label{T}
\end{equation}

%%%%%%%%%%%%%%%%
\begin{lemma}\label{improved convergence}
For $-2\beta< m\leq 1$ and $T>0$, $(c_{n})_{n\ge 1}$ converges to $c$ in  $C\big([0,T];w-L_m^1(0,\infty)\big)$.
\end{lemma}
%%%%%%%%%%%%%%%%

\begin{proof}
Let $T>0$, $t\in [0,T]$, and $\phi\in L^\infty(0,\infty)$. For $R\geq 1$, let us consider the following term
	\begin{align}
		&\left| \int_{0}^{\infty} v^m \phi(v)\left(c_n(t,v)-c(t,v)\right) dv\right| \nonumber \\
		&\leq  \int_{0}^{\frac{1}{R}} v^{m+2\beta} v^{-2\beta}|\phi(v)|\left(c_n(t,v)+c(t,v)\right) dv \nonumber \\
		\quad &+ \left| \int_{\frac{1}{R}}^{R} v^m \phi(v)\left(c_n(t,v)-c(t,v)\right) dv\right| \nonumber \\
		\quad &+ \int_{R}^{\infty} v |\phi(v)|(c_n(t,v)+c(t,v)) dv. \nonumber 
	\end{align}
By \Cref{Global prop-1} and~\eqref{M}, we have
	\begin{align*}
	&\left| \int_{0}^{\infty} v^m \phi(v)\left(c_n(t,v)-c(t,v)\right) dv\right|  \\
	&\leq \left| \int_{\frac{1}{R}}^{R} v^m \phi(v)\left(c_n(t,v)-c(t,v)\right) dv\right| + C(T) \left(\frac{1}{R}\right)^{m+2\beta} \|\phi\|_{L^{\infty}(0,\infty)}  \\
	 &\quad  + \left[ \int_R^\infty v c(t,v) dv + \sup_{n\ge 1} \int_R^\infty v c_n(t,v) dv \right]  \|\phi\|_{L^{\infty}(0,\infty)}. 
	\end{align*}
Taking the supremum with respect to $t$ on both sides over $[0,T]$, we get
\begin{align*}
 &\sup_{t\in [0,T]} \left| \int_{0}^{\infty} v^m \phi(v)\left(c_n(t,v)-c(t,v)\right) dv\right|  \\
 &\leq	\sup_{t\in [0,T]} \left| \int_{\frac{1}{R}}^{R} v^m \phi(v)\left(c_n(t,v)-c(t,v)\right) dv\right| + C(T) \left(\frac{1}{R}\right)^{m+2\beta} \|\phi\|_{L^{\infty}(0,\infty)}  \\
&\quad  + \left[ \sup_{t\in [0,T]} \int_R^\infty v c(t,v) dv + \sup_{n\ge 1} \sup_{t\in [0,T]} \int_R^\infty v c_n(t,v) dv \right]  \|\phi\|_{L^{\infty}(0,\infty)}. 
\end{align*}
Now taking the limit $n\rightarrow \infty$ and applying \eqref{convergence-3} yield
\begin{align*}
&\limsup_{n\rightarrow \infty} \sup_{t\in [0,T]} \left| \int_{0}^{\infty} v^m \phi(v)\left(c_n(t,v)-c(t,v)\right) dv\right|  \\
 &\leq	 C(T) \left(\frac{1}{R}\right)^{m+2\beta} \|\phi\|_{L^{\infty}(0,\infty)}  \\
 &\quad  + \left[ \sup_{t\in [0,T]} \int_R^\infty v c(t,v) dv + \sup_{n\ge 1} \sup_{t\in [0,T]} \int_R^\infty v c_n(t,v) dv \right]  \|\phi\|_{L^{\infty}(0,\infty)}. 
\end{align*}
Next, letting $R\rightarrow \infty$ and using \Cref{Tail-2} and~\eqref{T}, we obtain
 \begin{equation*}
 	\lim_{n\rightarrow \infty} \sup_{t\in [0,T]} \left| \int_{0}^{\infty} v^m \phi(v)\left(c_n(t,v)-c(t,v)\right) dv\right|=0,
 \end{equation*}
and thus complete the proof.
\end{proof}

At this stage, we are prepared to conclude the proof of \Cref{main theorem}. 

\begin{proof}[Proof of \Cref{main theorem}]
Let $T>0$. We have already shown in~\eqref{M} that $c\in L^\infty(0,T;L_m^1(0,\infty))$ for all $m\in [-2\beta,1]$, while \Cref{improved convergence} ensures that $c\in C([0,T];L_m^1(0,\infty))$ for all $m\in (-2\beta,1]$. Finally, the assumptions~\eqref{coagassum} and~\eqref{growthassum}, along with the convergences stated in \Cref{improved convergence}, allows us to use classical arguments, see \cite{BLL_book, Stewart1989} for instance, to pass to the limit as $n\to \infty$ in~\eqref{truncated weak formulation} and obtain that $c$ satisfies~\eqref{weak formulation}. 
\end{proof}

%%%%%%%%%%%%%%%%
\begin{corollary}\label{existenceofsumkernel}
		If all the assumptions made in \Cref{uniqueness theorem} are satisfied, then there exists at least one global weak solution $c$ to~\eqref{eq:1.3}--\eqref{eq:1.4}  in the sense of  \Cref{definitionofsolution} such that \eqref{solsecondmomentbound} holds true.
\end{corollary}
%%%%%%%%%%%%%%%%

\begin{proof}
	Existence of at least one weak solution $c$ satisfying~\eqref{solsecondmomentbound} follows directly from  \Cref{main theorem} and \Cref{second moment lemma-1}. 
\end{proof}

%%%%%%%%%%%%%%%%
%%%%%%%%%%%%%%%%
\section{Uniqueness}
%%%%%%%%%%%%%%%%
%%%%%%%%%%%%%%%%

This section is devoted to the proof of \Cref{uniqueness theorem}. Since the existence of at least one weak solution to~\eqref{eq:1.3}--\eqref{eq:1.4} as stated in \Cref{uniqueness theorem} follows from \Cref{existenceofsumkernel}, we are left with the uniqueness issue. In the subsequent study, we use \emph{solution to \eqref{eq:1.3}--\eqref{eq:1.4}} to mean the \emph{solution in the sense of \Cref{uniqueness theorem}}. We now focus on the key element of the uniqueness result, as presented in \Cref{uniqueness theorem}. The uniqueness result of \Cref{uniqueness theorem} is a consequence of the following continuous dependence result.

%%%%%%%%%%%%%%%%
\begin{proposition}\label{continuousdependence result} 
Assume that $K$ and $g$ satisfy~\eqref{coagassum} and~\eqref{growthassum}. Let $c_{1,0}$ and $c_{2,0}$ be two non-negative functions in $L_{-2\beta,2}^1(0,\infty)$. If $T>0$ and $c_1$ and $c_2$ are two weak solutions to \eqref{eq:1.3}--\eqref{eq:1.4} on $[0,T]$ with respective initial conditions $c_{1,0}$ and $c_{2,0}$ such that
\begin{equation}
	c_i \in L^\infty(0,T;L_{-2\beta,2}^1(0,\infty)), \quad i=1,2, \label{z014}
\end{equation}
then there exists $C(T)>0$ depending only on $k$, $A$, and the $L^\infty(0,T;L_{-2\beta,2}^1(0,\infty))$-norms of $c_1$ and $c_2$ such that
\begin{equation*}%\label{continuousdependence}
	\|c_{1}(t)-c_{2}(t)\|_{L_{-\beta,1}^{1}(0,\infty)}\leq C(T)	\|c_{1,0}-c_{2,0}\|_{L_{-\beta,1}^{1}(0,\infty)}, \quad t\in [0,T].
\end{equation*}
\end{proposition}
%%%%%%%%%%%%%%%%

We split the proof of \Cref{continuousdependence result} into several steps and first establish a differential inequality for a weight $L^1$-norm of $c_1-c_2$ in \Cref{lem.diffineq}. More precisely, introducing
\begin{align*}
	& E := c_1 - c_2,\quad E_{0} := E(0,\cdot) = c_{1,0} - c_{2,0}, \\
	& S := Q(c_{1}) - Q(c_{2}) \quad \text{and}\quad \nu(t):=\int_{0}^{\infty} w(v) |E(t,v)| dv, 
\end{align*}
for $(t,v)\in [0,T]\times (0,v)$, along with the weight
\begin{equation*}
	w(v) := v^{-\beta}+v, \quad v\in (0,\infty), %\label{weight}
\end{equation*}
we shall establish the following inequality.

%%%%%%%%%%%%%%%%
\begin{lemma}\label{lem.diffineq}
	There is $L_0>0$ depending only on $A$ such that
	\begin{equation}
		\nu(t) \le \nu(0) + L_0 \int_0^t \nu(s) ds + \int_0^t \int_0^\infty w(v) S(s,v)\, \mathrm{sign}(E(s,v)) dvds, \quad t\in [0,T]. \label{z001}
	\end{equation}
\end{lemma}
%%%%%%%%%%%%%%%%

Observe that the integrability properties of $c_1$ and $c_2$ ensure that $\nu(t)$ is well-defined for all $t\in [0,T]$, while the finiteness of the last term on the right-hand side of~\eqref{z001} is provided by the next result.

%%%%%%%%%%%%%%%%
\begin{lemma}\label{lem.Swd}
	$E\in L^\infty(0,T;L_{-2\beta,2}^1(0,\infty))$ and $S\in L^\infty(0,T;L_{-\beta,1}^1(0,\infty))$.
\end{lemma}
%%%%%%%%%%%%%%%%

\begin{proof}
The stated integrability properties of $E$ readily follow from~\eqref{z014}. 

Next, let $i\in\{1,2\}$ and $t\in [0,T]$. It follows from~\eqref{coagassum} and Fubini's theorem that
\begin{align*}
	\int_0^\infty w(v) |Q(c_i)(t,v)|dv & \le \frac{1}{2} \int_0^\infty \int_0^\infty w(v_1+v_2) K(v_1,v_2) c_i(t,v_1) c_i(t,v_2)dv_2dv_1 \\
	& \quad + \int_0^\infty \int_0^\infty w(v_1) K(v_1,v_2) c_i(t,v_1) c_i(t,v_2)dv_2dv_1 \\
	& \le 2 \int_0^\infty \int_0^\infty [w(v_1)+w(v_2)] K(v_1,v_2) c_i(t,v_1) c_i(t,v_2)dv_2dv_1 \\
	& \le 4k \int_0^1 \int_0^1 \left( v_1^{-\beta} + v_2^{-\beta} \right) (v_1v_2)^{-\beta} c_i(t,v_1) c_i(t,v_2)dv_2dv_1 \\
	& \quad + 8k \int_0^1 \int_1^\infty \left( v_1^{-\beta} + v_2 \right) v_1^{-\beta} v_2 c_i(t,v_1) c_i(t,v_2)dv_2dv_1 \\
	& \quad + 4k \int_1^\infty \int_1^\infty \left( v_1 + v_2 \right)^2 c_i(t,v_1) c_i(t,v_2)dv_2dv_1 \\
	& \le 8k M_{-2\beta}(c_i(t)) M_{-\beta}(c_i(t)) + 8k M_{-2\beta}(c_i(t)) M_{1}(c_i(t)) \\
	& \quad + 8k M_{-\beta}(c_i(t)) M_{2}(c_i(t)) + 16k M_2(c_i(t)) M_0(c_i(t)), 
\end{align*}
and we infer from the integrability properties~\eqref{z014} of $c_i$ that the right-hand side of the above inequality belongs to $L^\infty(0,T)$. Therefore, $Q(c_i)$ belongs to $L^\infty(0,T;L_{-\beta,1}^1(0,\infty))$ for $i\in\{1,2\}$, and so does $S=Q(c_1)-Q(c_2)$.
\end{proof}

Formally, \Cref{lem.diffineq} follows from the multiplication of the equation solved by $E$ (derived from~\eqref{eq:1.3} for $c_1$ and $c_2$) by $w\,\mathrm{sign}(E)$ and integration with respect to both volume and time. However, since $c_1-c_2$ need not be differentiable, a suitable regularization in the spirit of DiPerna \& Lions theory for transport equations \cite{diperna1989ordinary} is used, see also \cite[Appendix~6.1 \&~6.2]{perthame2006transport}. Specifically, we first extend $E$, $S$, and $g$ to $[0,T]\times\mathbb{R}$ by zero and define their extensions $\tilde{E}$, $\tilde{S}$, and $\tilde{g}$ as follows: for $t\in [0,T]$, 
\begin{align*}
	\tilde{E}(t,v):=\begin{cases}
		E(t,v)\ &\text{ for }\ v\in (0,\infty),\\
		0\ &\text{ for }\ v\in (-\infty,0),
	\end{cases}\quad 
	\tilde{E_{0}}(v):=\begin{cases}
	E(0,v)\ &\text{ for }\ v\in (0,\infty),\\
	0\ &\text{ for }\ v\in (-\infty,0),
\end{cases} 
\end{align*}
\begin{align*}
		\tilde{S}(t,v):=\begin{cases}
		S(t,v)\ &\text{ for }\ v\in (0,\infty),\\
		0\ &\text{ for }\ v\in (-\infty,0),
	\end{cases}\quad
	\tilde{g}(t,v):=\begin{cases}
		g(t,v)\ &\text{ for }\ v\in (0,\infty),\\
		0\ &\text{ for }\ v\in (-\infty,0). 
	\end{cases}
\end{align*}
It readily follows from~\eqref{growthassum} that $\tilde{g}$ is weakly differentiable with respect to $v$ on $\mathbb{R}$ and, for $t\in [0,T]$, $\partial_v \tilde{v}(t)$ is given by 
\begin{align*}
\partial_v \tilde{g}(t,v) =
	\begin{cases}
		0 & \text{ for } v \in (-\infty, 0), \\
		\partial_v g(t,v) & \text{ for } v \in (0, \infty).
	\end{cases}
\end{align*}
In particular, $\tilde{g}(t,.)\in W^{1,\infty}(I)$ for any bounded interval  $I\subset\mathbb{R}$ for every $t\in [0,T]$. This implies that $g(t,.)$ is absolutely continuous on $I$ for every $t\in [0,T]$ and satisfies 
\begin{align*}
	\tilde{g}(t,v_{2})-\tilde{g}(t,v_{1})=\int_{v_{1}}^{v_{2}} \partial_v \tilde{g}(t,v)dv, \quad (v_1,v_2)\in\mathbb{R}^2.
\end{align*}
 Moreover, from \eqref{growthassum}, we obtain
\begin{equation}\label{growthtilde}
	|\partial_v \tilde{g}(t,v)|\leq A \quad \text{and}\quad  |\tilde{g}(t,v)|\leq A|v|, \quad (t,v)\in [0,T]\times\mathbb{R}.
\end{equation}
Next, given a regularizing sequence of mollifiers $\rho_{\delta}(v):=\frac{1}{\delta}\rho(\frac{v}{\delta})$, $v\in\mathbb{R}$, with $\delta \in (0,1)$, $\rho\in \mathcal{D}(\mathbb{R})$, supp($\rho$) $\subset (-1,1)$, $\rho\geq 0$ and  $\|\rho\|_{L^{1}}=1$, we put 
	\begin{align*}
		\tilde{E}^{\delta}:=\tilde{E}\star \rho_{\delta},\quad 	\tilde{E_{0}}^{\delta}:=\tilde{E_{0}}\star \rho_{\delta}\quad \text{and}\quad \tilde{S}^{\delta}:=\tilde{S}\star \rho_{\delta}.
\end{align*}

For further use, we report the following properties of the convolution in weighted $L^1$-spaces.

%%%%%%%%%%%%%%%%
\begin{lemma}\label{lem.convolution}
	Let $m\ge 0$ and $\varphi\in L^1(\mathbb{R},(1+|v|^m)dv)$. For every $\delta\in (0,1)$, the $C^\infty$-smooth function $\rho_\delta\star\varphi$ belongs to $L^1(\mathbb{R},(1+|v|^m)dv)$ with
	\begin{equation*}
		\int_{\mathbb{R}} |(\rho_\delta\star\varphi)(v)| (1+|v|^m)dv \le 2^{m+1} \int_{\mathbb{R}} |\varphi(v)| (1+|v|^m)dv.
	\end{equation*}
	Moreover,
	\begin{equation*}
		\lim_{\delta\to 0} \int_{\mathbb{R}} |(\rho_\delta\star\varphi -\varphi)(v)| (1+|v|^m)dv = 0.
	\end{equation*}
\end{lemma}
%%%%%%%%%%%%%%%%

\Cref{lem.convolution} is a classical property of the convolution in $\mathbb{R}$ for $m=0$ and the proof for $m>0$ relies on the splitting of the contribution of the interval $(-R,R)$ on which the classical result for $m=0$ can be applied and a control of the tails on $(-\infty,R)\cup (R,\infty)$ which is due to the integrability properties of $\varphi$ and the elementary inequality $(|v_1|+|v_2|)^m \le 2^m (|v_1|^m + |v_2|^m)$ for $(v_1,v_2)\in \mathbb{R}^2$. We omit the proof.  

We next report an immediate consequence of \Cref{lem.Swd} and \Cref{lem.convolution}.

%%%%%%%%%%%%%%%%
\begin{corollary}\label{cor.convdelta}
	For each $\delta\in (0,1)$, $\tilde{E}^\delta \in L^\infty(0,T;L^1(\mathbb{R},(1+v^2)dv))$ and $\tilde{S}^\delta \in L^\infty(0,T;L^1(\mathbb{R},(1+|v|)dv))$.
\end{corollary}
%%%%%%%%%%%%%%%%

Next, according to the definition of $\tilde{E}$, $\tilde{S}$, and $\tilde{g}$, it readily follows from \Cref{definitionofsolution} that, for $t\in [0,T]$ and $\phi\in W^{1,\infty}(\mathbb{R})$, 
\begin{align*}
	\int_{\mathbb{R}} \tilde{E}(t,v_1)\phi(v_1)dv_1 & = \int_{\mathbb{R}} \tilde{E}_0(v_1)\phi(v_1)dv_1 + \int_0^t \int_{\mathbb{R}} \phi'(v_1) \tilde{g}(s,v_1) \tilde{E}(s,v_1)dv_1ds \\
	& \quad + \int_0^t \int_{\mathbb{R}} \phi(v_1) \tilde{S}(s,v_1)\phi(v_1)dv_1ds.
\end{align*}
For $(t,v)\in [0,T]\times\mathbb{R}$, we take $\phi(v_1)=\rho_\delta(v-v_1)$ in the above identity to deduce that $\tilde{E}^\delta$ is a smooth solution to 
\begin{equation*}
	\partial_t \tilde{E}^\delta + \partial_v \left( \rho_\delta\star \big(\tilde{g}\tilde{E}\big)\right) = \tilde{S}^\delta , \quad (t,v)\in (0,T)\times\mathbb{R}.
\end{equation*}
Equivalently,
\begin{equation}
	\partial_t \tilde{E}^\delta + \partial_v \left( \tilde{g}\tilde{E}^\delta \right) = D^\delta + \tilde{S}^\delta , \quad (t,v)\in (0,T)\times\mathbb{R}, \label{moleq}
\end{equation}
where
\begin{equation*}
	D^\delta :=  \partial_v \left( \tilde{g}\tilde{E}^\delta - \rho_\delta\star \big(\tilde{g}\tilde{E}\big) \right).
\end{equation*}

Next, for $\epsilon\in (0,1)$, we define 
\begin{equation*}
	\Sigma_{\epsilon}(z):=\frac{z^{2}}{\sqrt{z^{2}+\epsilon}}, \quad z\in\mathbb{R}.
\end{equation*}
Clearly, for every $\epsilon \in (0,1)$, $\Sigma_{\epsilon}$ is a continuously differentiable function satisfying the following properties: 
\begin{align}
	& \big| \Sigma_{\epsilon}(z) - |z| \big| \le \min\{\sqrt{\epsilon},|z|\}, \quad \big| \Sigma_{\epsilon}'(z) - \mathrm{sign}(z) \big| \le \frac{2\epsilon}{z^2+\epsilon} \le 2\quad \text{and}\quad \big|\Sigma_{\epsilon}'(z)\big|\le 2,\quad z\in \mathbb{R}. \label{sigmapointwiseconvergence}\\
	& \big| z \Sigma_{\epsilon}'(z) - \Sigma_\epsilon(z) \big| \le \min\{\sqrt{\epsilon},|z|\}, \quad  z\in \mathbb{R}. \label{sigmabound}	
\end{align}

Using the chain rule, we obtain the following equation from~\eqref{moleq}
\begin{equation}
	\partial_t	\Sigma_{\epsilon}(\tilde{E}^{\delta}) +	\partial_v\left(\tilde{g}\Sigma_{\epsilon}(\tilde{E}^{\delta})\right) = \Sigma_{\epsilon}'(\tilde{E}^{\delta}) D^{\delta} + \Sigma_{\epsilon}'(\tilde{E}^{\delta}) \tilde{S}^{\delta} - \partial_v \tilde{g} \left[ \tilde{E}^{\delta} \Sigma_{\epsilon}'(\tilde{E}^{\delta}) - \Sigma_{\epsilon}(\tilde{E}^{\delta}) \right] \label{z004}
\end{equation}
for $(t,v)\in (0,T)\times\mathbb{R}$. Now, setting $w_\epsilon(v) := \sqrt{v^2+\epsilon} + (v^2+\epsilon)^{-\beta/2}$ for $v\in\mathbb{R}$, we note that 
	\begin{equation}
		0 \le w_\epsilon(v) \le |v| + \sqrt{\epsilon} + |v|^{-\beta} \le \big( 1 + \sqrt{\epsilon} \big) w(|v|), \quad v\in\mathbb{R}\setminus\{0\}, \label{z002}
	\end{equation}
and
	\begin{align}
		\big| w_\epsilon(v_1) - w_\epsilon(v_2) \big| & \le \left| \int_{v_1}^{v_2} \left[ \frac{v}{\sqrt{v^2+\epsilon}} - \beta\frac{v}{(v^2+\epsilon)^{(\beta+2)/2}} \right]dv \right| \nonumber \\
		& \le \left( 1 + \beta \epsilon^{-(\beta+1)/2} \right) |v_1-v_2| \label{z003}
	\end{align}
for $(v_1,v_2)\in \mathbb{R}^2$. In addition,
\begin{equation}
	w_\epsilon(v) \le 2 \epsilon^{-\beta/2} (1+|v|), \qquad v\in\mathbb{R}. \label{z006}
\end{equation}

Before going on, let us check that all terms in~\eqref{z004} belong to the appropriate space.

%%%%%%%%%%%%%%%%
\begin{lemma}\label{convoestimate}
There is $C(T)>0$ depending on $K$, $g$, $c_1$ and $c_2$ such that, for $t\in [0,T]$, 
\begin{subequations}\label{z005}
	\begin{align}
		\int_{\mathbb{R}} (1+|v|) \left| \partial_v\left(\tilde{g}(t,v) \Sigma_{\epsilon}(\tilde{E}^{\delta}(t,v))\right) \right| dv & \le \frac{C(T)}{\delta}, \label{z005a}\\
		\int_{\mathbb{R}} (1+|v|) \left| \Sigma_{\epsilon}'(\tilde{E}^{\delta}(t,v)) D^{\delta}(t,v) \right| dv & \le \frac{C(T)}{\delta}, \label{z005b}\\
		\int_{\mathbb{R}} (1+|v|) \left| \Sigma_{\epsilon}'(\tilde{E}^{\delta}(t,v)) \tilde{S}^{\delta}(t,v) \right| dv & \le C(T), \label{z005c} \\
		\int_{\mathbb{R}} (1+|v|) \left| \partial_v \tilde{g}(t,v) \left[ \tilde{E}^{\delta} \Sigma_{\epsilon}'(\tilde{E}^{\delta}) - \Sigma_{\epsilon}(\tilde{E}^{\delta}) \right](t,v)  \right| dv & \le C(T), \label{z005d} \\
		\int_{\mathbb{R}} (1+|v|) \left| 	\partial_t	\Sigma_{\epsilon}(\tilde{E}^{\delta}(t,v)) \right| dv & \le \frac{C(T)}{\delta}. \label{z005e} 
	\end{align}
\end{subequations}
\end{lemma}	
%%%%%%%%%%%%%%%%
	
\begin{proof}
	We begin with the proof of~\eqref{z005a} and~\eqref{z005b}. Since
	\begin{equation*}
		\partial_v\left(\tilde{g}\Sigma_{\epsilon}(\tilde{E}^{\delta})\right) = \Sigma_{\epsilon}(\tilde{E}^{\delta}) \partial_v\tilde{g}  + \tilde{g} \Sigma_{\epsilon}'(\tilde{E}^{\delta}) \big(\rho_\delta'\star\tilde{E}\big) %\label{z007}
	\end{equation*}
	and
	\begin{equation}
		D^\delta = \tilde{g} \big(\rho_\delta'\star\tilde{E}\big) + \tilde{E}^\delta \partial_v\tilde{g} - \rho_\delta'\star\big(\tilde{g}\tilde{E}\big), \label{z008}
	\end{equation}
	it suffices to study the integrability of each term separately. First, owing to~\eqref{growthtilde}, \eqref{sigmapointwiseconvergence}, and \Cref{lem.convolution} (with $m=1$), we obtain
	\begin{align*}
		\int_{\mathbb{R}} (1+|v|) \left| \Sigma_{\epsilon}(\tilde{E}^{\delta}(v)) \partial_v\tilde{g}(v) \right| dv & \le 2 \int_{\mathbb{R}} (1+|v|) \left| \tilde{E}^{\delta}(v) \partial_v\tilde{g}(v) \right| dv \le 2A \int_{\mathbb{R}} (1+|v|) \left| \tilde{E}^{\delta}(v) \right| dv \\ 
		& \le 8A \int_{\mathbb{R}} (1+|v|) \left| \tilde{E}(v) \right| dv = 8A \int_0^\infty (1+v) |E(v)|dv.
	\end{align*}
	We next infer from~\eqref{growthtilde}, \eqref{sigmapointwiseconvergence}, and Fubini's theorem that
	\begin{align*}
		\int_{\mathbb{R}} (1+|v|) \left| \tilde{g}(v) \Sigma_{\epsilon}'(\tilde{E}^{\delta}(v)) \big(\rho_\delta'\star\tilde{E}\big)(v) \right|dv & \le 2 \int_{\mathbb{R}} (1+|v|) \left| \tilde{g}(v) \big(\rho_\delta'\star\tilde{E}\big)(v) \right|dv \\
		& \le 2A \int_{\mathbb{R}}  \int_{\mathbb{R}} |v| (1+|v|) |\rho_\delta'(v_1)| |\tilde{E}(v-v_1)| dv_1dv \\
		& \le \frac{2A}{\delta} \int_{\mathbb{R}} \int_{\mathbb{R}} (1+|v+\delta v_1|)^2  |\rho'(v_1)| |\tilde{E}(v)| dv_1dv \\
		& \le \frac{2A}{\delta} \|\rho'\|_{L^1(\mathbb{R})} \int_{\mathbb{R}} (2+|v|)^2  |\tilde{E}(v)| dv \\
		& \le \frac{16A}{\delta} \|\rho'\|_{L^1(\mathbb{R})} \int_0^\infty (1+v^2)  |E(v)| dv
	\end{align*}
	and
	\begin{align*}
		 \int_{\mathbb{R}} (1+|v|) \left| \rho_\delta'\star\big(\tilde{g}\tilde{E}\big)(v) \right|dv & \le A \int_{\mathbb{R}}  \int_{\mathbb{R}} |v-v_1| (1+|v|) |\rho_\delta'(v_1)| |\tilde{E}(v-v_1)| dv_1dv \\
		& \le \frac{A}{\delta} \int_{\mathbb{R}} \int_{\mathbb{R}} (1+|v+\delta v_1|) |v|  |\rho'(v_1)| |\tilde{E}(v)| dv_1dv \\
		& \le \frac{8A}{\delta} \|\rho'\|_{L^1(\mathbb{R})} \int_0^\infty (1+v^2)  |E(v)| dv.
	\end{align*}
Collecting the above estimates and recalling~\eqref{z014} give~\eqref{z005a} and~\eqref{z005b}.

Next, the estimate~\eqref{z005c} readily follows from \Cref{lem.Swd} and \Cref{lem.convolution} while~\eqref{z005d} is a straightforward consequence of~\eqref{z014}, \eqref{growthtilde}, \eqref{sigmabound}, and \Cref{lem.convolution}.  Finally, we deduce~\eqref{z005e} from~\eqref{z004}, \eqref{z005a}, \eqref{z005b}, \eqref{z005c}, and~\eqref{z005d}.
\end{proof}	

\begin{proof}[Proof of \Cref{lem.diffineq}]
We multiply both sides of the above equation by $w_\epsilon$ and integrate with respect to $v$ on $\mathbb{R}$ to obtain
\begin{align*}
	& \frac{d}{dt} \int_{\mathbb{R}} w_\epsilon \Sigma_{\epsilon}(\tilde{E}^{\delta})dv + \int_{\mathbb{R}} w_\epsilon \partial_v \left( \tilde{g} \Sigma_{\epsilon}(\tilde{E}^{\delta}) \right) dv\\
	& \quad = \int_{\mathbb{R}} w_\epsilon \left(\Sigma_{\epsilon}'(\tilde{E}^{\delta}) D^{\delta} + \Sigma_{\epsilon}'(\tilde{E}^{\delta}) \tilde{S}^{\delta} + \partial_v \tilde{g} \left[ \Sigma_{\epsilon}(\tilde{E}^{\delta}) - \tilde{E}^{\delta}\Sigma_{\epsilon}'(\tilde{E}^{\delta}) \right] \right)dv,
\end{align*}
observing that all the terms appearing in the above equation are well-defined according to \Cref{convoestimate} and~\eqref{z006}.  Now, applying the integration by parts formula to the second term of the left-hand side of the above equation and observing that the boundary terms vanish by~\eqref{growthtilde} and \Cref{cor.convdelta}, we obtain
\begin{align*}
	&\frac{d}{dt} \int_{\mathbb{R}} w_\epsilon \Sigma_{\epsilon}(\tilde{E}^{\delta})dv - \int_{\mathbb{R}} w_\epsilon' \tilde{g} \Sigma_\epsilon(\tilde{E}^\delta) dv\\
	& \qquad = \int_{\mathbb{R}} w_\epsilon \left( \Sigma_{\epsilon}'(\tilde{E}^{\delta}) D^{\delta} + \Sigma_{\epsilon}'(\tilde{E}^{\delta}) \tilde{S}^{\delta} + \partial_v \tilde{g} \left[ \Sigma_{\epsilon}(\tilde{E}^{\delta}) - \tilde{E}^{\delta} \Sigma_{\epsilon}'(\tilde{E}^{\delta}) \right] \right)dv.
\end{align*}
Integrating with respect to time over $(0,t)$ for $t\in [0,T]$ gives
\begin{equation}
	\begin{split}
		 \int_{\mathbb{R}} w_\epsilon \Sigma_{\epsilon}(\tilde{E}^{\delta}(t))dv & =  \int_{\mathbb{R}} w_\epsilon \Sigma_{\epsilon}(\tilde{E}^{\delta}(0))dv + \int_0^t \int_{\mathbb{R}} w_\epsilon' \tilde{g} \Sigma_\epsilon(\tilde{E}^\delta) dvds \\
		 & \quad + \int_0^t \int_{\mathbb{R}} w_\epsilon \left( \Sigma_{\epsilon}'(\tilde{E}^{\delta}) D^{\delta} + \Sigma_{\epsilon}'(\tilde{E}^{\delta}) \tilde{S}^{\delta} + \partial_v \tilde{g} \left[ \Sigma_{\epsilon}(\tilde{E}^{\delta}) - \tilde{E}^{\delta} \Sigma_{\epsilon}'(\tilde{E}^{\delta}) \right] \right)dvds.
	\end{split}\label{epdelmain-1}
\end{equation}
At this point, we note that, by~\eqref{growthtilde} and~\eqref{sigmapointwiseconvergence},
\begin{align*}
	I_1 & := \left| \int_{\mathbb{R}} w_\epsilon'(v) \tilde{g}(v) \Sigma_\epsilon(\tilde{E}^\delta(v)) dv \right| \le 2A \int_{\mathbb{R}} |v w_\epsilon'(v)| \left| \tilde{E}^\delta(v) \right| dv \\
	& \le 2A(1+\beta) \int_{\mathbb{R}} w_\epsilon(v) \left| \tilde{E}^\delta(v) \right| dv.
\end{align*}
Since
\begin{align}
	\int_{\mathbb{R}} w_\epsilon(v) \left| \tilde{E}^\delta(v) \right| dv & \le \int_{\mathbb{R}} \int_{\mathbb{R}} w_\epsilon(v) \rho_\delta(v_1) \left| \tilde{E}(v-v_1) \right| dv_1dv \nonumber\\
	& = \int_{\mathbb{R}} \int_{\mathbb{R}} w_\epsilon(v+\delta v_1) \rho(v_1) \left| \tilde{E}(v) \right| dv_1dv \nonumber \\
	& \le \int_{\mathbb{R}} \left[ w_\epsilon(v) + \left( 1 + \beta \epsilon^{-(\beta+1)/2} \right)\delta \right] \left| \tilde{E}(v) \right| dv \label{z009}
\end{align}
by~\eqref{z003}, we conclude that
\begin{equation}
	I_1 \le 2A(1+\beta) \int_{\mathbb{R}} \left[ w_\epsilon(v) + \left( 1 + \beta \epsilon^{-(\beta+1)/2} \right)\delta \right] \left| \tilde{E}(v) \right| dv. \label{z010}
\end{equation}
Next, recalling~\eqref{z008}, we infer from~\eqref{growthtilde} and~\eqref{sigmapointwiseconvergence} that
\begin{align*}
	I_2 & := \left| \int_{\mathbb{R}} w_\epsilon(v) \Sigma_{\epsilon}'(\tilde{E}^{\delta}(v)) D^{\delta}(v) dv \right| \\& 
	\le 2 \int_{\mathbb{R}} w_\epsilon(v) \left| \tilde{g}(v) \big(\rho_\delta'\star\tilde{E}\big)(v) - \rho_\delta'\star\big(\tilde{g}\tilde{E}\big)(v) \right| dv + 2A \int_{\mathbb{R}} w_\epsilon(v) \left| \tilde{E}^{\delta}(v) \right| dv.
\end{align*}
Now, thanks to~\eqref{growthtilde} and~\eqref{z003},
\begin{align*}
	& \int_{\mathbb{R}} w_\epsilon(v) \left| \tilde{g}(v) \big(\rho_\delta'\star\tilde{E}\big)(v) - \rho_\delta'\star\big(\tilde{g}\tilde{E}\big)(v) \right| dv \\
	& \quad \le \int_{\mathbb{R}} w_\epsilon(v) \left| \int_{\mathbb{R}} \left[  \tilde{g}(v) \rho_\delta'(v_1) \tilde{E}(v-v_1) - \rho_\delta'(v_1) \tilde{g}(v-v_1) \tilde{E}(v-v_1) \right]dv_1 \right| dv \\
	& \quad \le \int_{\mathbb{R}} \int_{\mathbb{R}} w_\epsilon(v) \left|  \tilde{g}(v) - \tilde{g}(v-v_1) \right| |\rho_\delta'(v_1)| \left| \tilde{E}(v-v_1) \right| dv_1dv \\
	& \quad \le A \int_{\mathbb{R}} \int_{\mathbb{R}} w_\epsilon(v) |v_1 \rho_\delta'(v_1)| \left| \tilde{E}(v-v_1) \right| dv_1dv = A \int_{\mathbb{R}} \int_{\mathbb{R}} w_\epsilon(v+\delta v_1) |v_1 \rho'(v_1)| \left| \tilde{E}(v) \right| dv_1dv \\
	& \quad \le A I(\rho) \int_{\mathbb{R}} \left[ w_\epsilon(v) + \left( 1 + \beta \epsilon^{-(\beta+1)/2} \right)\delta \right] \left| \tilde{E}(v) \right| dv,
\end{align*}
with 
\begin{equation*}
	I(\rho) := \int_{\mathbb{R}} |v \rho'(v)| dv.
\end{equation*}
Combining the above two estimates with~\eqref{z009} leads us to 
\begin{equation}
	I_2 \le (2+I(\rho))A \int_{\mathbb{R}} \left[ w_\epsilon(v) + \left( 1 + \beta \epsilon^{-(\beta+1)/2} \right)\delta \right] \left| \tilde{E}(v) \right| dv. \label{z011}
\end{equation}
We next use once more~\eqref{growthtilde} and~\eqref{z009}, along with~\eqref{sigmabound}, to obtain
\begin{align}
	I_3 & := \left| \int_{\mathbb{R}} w_\epsilon(v) \partial_v \tilde{g}(v) \left[ \Sigma_{\epsilon}(\tilde{E}^{\delta}(v)) - \tilde{E}^{\delta}(v) \Sigma_{\epsilon}'(\tilde{E}^{\delta}(v)) \right]dv \right| \nonumber \\
	& \le A \int_{\mathbb{R}} w_\epsilon(v) \left| \tilde{E}^{\delta}(v) \right| dv \nonumber \\
	& \le A  \int_{\mathbb{R}} \left[ w_\epsilon(v) + \left( 1 + \beta \epsilon^{-(\beta+1)/2} \right)\delta \right] \left| \tilde{E}(v) \right| dv. \label{z012}
\end{align}
Collecting~\eqref{z010}, \eqref{z011}, and~\eqref{z012} and using again~\eqref{z009}, we infer from~\eqref{epdelmain-1} that there is $L_0>0$ depending only on $A$ and $I(\rho)$ such that
\begin{equation}
	\begin{split}
	\int_{\mathbb{R}} w_\epsilon \Sigma_\epsilon(\tilde{E}^\delta(t))dv & \le \int_{\mathbb{R}} \left[ w_\epsilon + \left( 1 + \beta \epsilon^{-(\beta+1)/2} \right)\delta \right] \left| \tilde{E}_0 \right| dv + \int_0^t \int_{\mathbb{R}} w_\epsilon \tilde{S}^\delta(s) \Sigma_\epsilon(\tilde{E}^\delta(s)) dvds \\
	& \quad + L_0 \int_0^t \int_{\mathbb{R}} \left[ w_\epsilon + \left( 1 + \beta \epsilon^{-(\beta+1)/2} \right)\delta \right] \left| \tilde{E}(s) \right| dvds.
	\end{split} \label{epdelmain-2}
\end{equation}

In order to pass the limit $\delta \rightarrow 0$ on both sides of~\eqref{epdelmain-2}, we consider each term separately. Using~\eqref{sigmapointwiseconvergence}, let us first consider the following term as
\begin{equation*}
	\left|\int_{\mathbb{R}} w_\epsilon(v) \left[	\Sigma_{\epsilon}(\tilde{E}^{\delta}(t,v)) - \Sigma_{\epsilon}(\tilde{E}(t,v)) \right] dv \right| \leq 2 \int_{\mathbb{R}} w_\epsilon(v) \left| \tilde{E}^{\delta}(t,v) - \tilde{E}(t,v) \right| dv. 
\end{equation*}
It readily follows from~\eqref{z006},  \Cref{lem.convolution}, and Lebesgue's dominated convergence theorem that we can let $\delta\rightarrow 0$ in the above inequality and find
\begin{equation}
	\lim_{\delta\to 0} \int_{\mathbb{R}} w_\epsilon(v) \Sigma_{\epsilon}(\tilde{E}^{\delta}(t,v)) dv = \int_{\mathbb{R}} w_\epsilon(v)  \Sigma_{\epsilon}(\tilde{E}(t,v)) dv = \int_{0}^{\infty} w_\epsilon(v)  \Sigma_{\epsilon}(E(t,v)) dv. \label{dellim-3}
\end{equation}
Similarly,
	\begin{align}
	&\left| \int_0^t \int_{\mathbb{R}} w_\epsilon(v) \left[ \Sigma_{\epsilon}'(\tilde{E}^{\delta}(s,v)) \tilde{S}^{\delta}(s,v) - \Sigma_{\epsilon}'(\tilde{E}(s,v)) \tilde{S}(s,v) \right] dvds \right|\nonumber\\
	&\leq \int_0^t \int_{\mathbb{R}} w_\epsilon(v) \left| \Sigma_{\epsilon}'(\tilde{E}^{\delta}(s,v)) \right| \left| \tilde{S}^{\delta}(s,v) - \tilde{S}(s,v)\right| dvds \nonumber\\
	& \quad + \int_0^t \int_{\mathbb{R}} w_\epsilon(v) \left| \tilde{S}(s,v) \right| \left| \Sigma_{\epsilon}'(\tilde{E}^{\delta}(s,v)) - \Sigma_{\epsilon}'(\tilde{E}(s,v)) \right| dvds.\label{coagepsilon-1}
\end{align}
On the one hand, by~\eqref{sigmapointwiseconvergence},
\begin{equation*}
	\int_0^t \int_{\mathbb{R}} w_\epsilon(v) \left| \Sigma_{\epsilon}'(\tilde{E}^{\delta}(s,v)) \right| \left| \tilde{S}^{\delta}(s,v) - \tilde{S}(s,v)\right| dvds \le 2 \int_0^t \int_{\mathbb{R}} w_\epsilon(v) \left| \tilde{S}^{\delta}(s,v) - \tilde{S}(s,v)\right| dvds,
\end{equation*}
and we use again~\eqref{z006},  \Cref{lem.convolution}, and Lebesgue's dominated convergence theorem, together with \Cref{lem.Swd} , to conclude that
\begin{equation}
	\lim_{\delta\to 0} \int_0^t \int_{\mathbb{R}} w_\epsilon(v) \left| \Sigma_{\epsilon}'(\tilde{E}^{\delta}(s,v)) \right| \left| \tilde{S}^{\delta}(s,v) - \tilde{S}(s,v)\right| dvds = 0.  \label{coagepsilon-2}
\end{equation}
On the other hand, $w_\epsilon \tilde{S}$ belongs to $L^1((0,T)\times\mathbb{R})$ thanks to \Cref{lem.Swd} and~\eqref{z002}, while \Cref{lem.convolution}, the Lipschitz continuity of $\Sigma_\epsilon'$, and~\eqref{sigmapointwiseconvergence} ensure that, at least for a subsequence (not relabeled),
\begin{equation*}
	\lim_{\delta\to 0} \left| \Sigma_{\epsilon}'(\tilde{E}^{\delta}(s,v)) - \Sigma_{\epsilon}'(\tilde{E}(s,v)) \right| = 0 \;\;\text{ for a.e. }\;\; (s,x)\in (0,t)\times\mathbb{R}
\end{equation*}
with
\begin{equation*}
	\left| \Sigma_{\epsilon}'(\tilde{E}^{\delta}(s,v)) - \Sigma_{\epsilon}'(\tilde{E}(s,v)) \right| \le 4 \;\;\text{ for a.e. }\;\; (s,x)\in (0,t)\times\mathbb{R}.
\end{equation*}
We are then in a position to apply Lebesgue's dominated convergence theorem and obtain
\begin{equation}
	\lim_{\delta\to 0} \int_0^t \int_{\mathbb{R}} w_\epsilon(v) \left| \tilde{S}(s,v) \right| \left| \Sigma_{\epsilon}'(\tilde{E}^{\delta}(s,v)) - \Sigma_{\epsilon}'(\tilde{E}(s,v)) \right| = 0.  \label{coagepsilon-3}
\end{equation}
Gathering~\eqref{coagepsilon-1}, \eqref{coagepsilon-2}, and~\eqref{coagepsilon-3}, we end up with
\begin{align}
	\lim_{\delta\to 0} \int_0^t \int_{\mathbb{R}} w_\epsilon(v) \Sigma_{\epsilon}'(\tilde{E}^{\delta}(s,v)) \tilde{S}^{\delta}(s,v) dvds & = \int_0^t \int_{\mathbb{R}} w_\epsilon(v) \Sigma_{\epsilon}'(\tilde{E}(s,v)) \tilde{S}(s,v) dvds \nonumber \\
	& = \int_0^t \int_{0}^\infty w_\epsilon(v) \Sigma_{\epsilon}'(E(s,v)) S(s,v) dvds. \label{coagepsilon-4}
\end{align}
Owing to~\eqref{dellim-3} and~\eqref{coagepsilon-4}, we can pass to the limit $\delta\rightarrow 0$ in~\eqref{epdelmain-2} and end up with
\begin{equation}
	\begin{split}
		\int_0^\infty w_\epsilon \Sigma_\epsilon(E(t)) dv & \le \int_{0}^\infty w_\epsilon \left| E_0 \right| dv + \int_0^t \int_{0}^\infty w_\epsilon S(s) \Sigma_\epsilon(E(s)) dvds \\
		& \quad + L_0 \int_0^t \int_{0}^\infty w_\epsilon \left| E(s) \right| dvds.
	\end{split} \label{epdelmain-3}
\end{equation}

We now perform the limit $\epsilon\rightarrow 0$ in~\eqref{epdelmain-3}. First, since 
\begin{equation*}
	\lim_{\epsilon\to 0} w_\epsilon(v) = w(v) = v + v^{-\beta}, \quad v\in (0,\infty),
\end{equation*}
and $E\in L^\infty(0,T; L_{-2\beta,2}^1(0,\infty))$ by \Cref{lem.Swd}, we readily infer from~\eqref{sigmapointwiseconvergence} and Lebesgue's dominated convergence theorem that
\begin{align*}
	\lim_{\epsilon\to 0} \int_0^\infty w_\epsilon(v) \Sigma_\epsilon(E(t,v)) dv & = \int_0^\infty w(v) |E(t,v)| dv, \\
	\lim_{\epsilon\to 0} \int_0^\infty w_\epsilon(v) \Sigma_\epsilon(E_0(v)) dv & = \int_0^\infty w(v) |E_0(v)| dv, \\
	\lim_{\epsilon\to 0} \int_0^t \int_0^\infty w_\epsilon(v) \Sigma_\epsilon(E(s,v)) dvds & = \int_0^t \int_0^\infty w(v) |E(s,v)| dvds. \\
\end{align*}
Next, we recall that
\begin{equation*}
	\lim_{\epsilon\to 0} w_\epsilon(v) S(s,v) \Sigma_\epsilon'(E(s,v)) = w_(v) S(s,v) \,\mathrm{sign}(E(s,v)) \;\;\text{ for a.e. }\;\; (s,x)\in (0,t)\times (0,\infty),
\end{equation*}
with
\begin{equation*}
	\left| w_\epsilon(v) S(s,v) \Sigma_\epsilon'(E(s,v)) \right| \le \left( v + \sqrt{\epsilon} + v^{-\beta} \right) |S(s,v)| \le 2 w(v) |S(s,v)|,
\end{equation*}
due to~\eqref{sigmapointwiseconvergence} and~\eqref{z002}. A further application of Lebesgue's dominated convergence theorem gives
\begin{equation*}
	\lim_{\delta\to 0} \int_0^t \int_{0}^\infty w_\epsilon S(s) \Sigma_\epsilon(E(s)) dvds = \int_0^t \int_{0}^\infty w S(s) \,\mathrm{sign}(E(s)) dvds.
\end{equation*}
We may then let $\epsilon\to 0$ in~\eqref{epdelmain-3} to obtain \Cref{lem.diffineq}.
\end{proof}

Owing to \Cref{lem.diffineq}, we are left with estimating the contribution of the coagulation term to complete the proof of the continuous dependence, a computation which is by now classical, see \cite[Section~8.2.5]{BLL_book} for instance, but which we sketch below for the sake of completeness. 
	
\begin{proof}[Proof of \Cref{continuousdependence result}]
Let us consider the last term on the right-hand side of~\eqref{z001}:
\begin{align*}
	&\int_{0}^{t}\int_{0}^{\infty} w(v)\, \mathrm{sign}(E(s,v)) S(s,v) dvds\nonumber\\
	& \quad = \frac{1}{2} \int_{0}^{t}\int_{0}^{\infty}\int_{0}^{\infty} K(v_{1},v_{2}) W(s,v_{1},v_{2}) (c_1+c_2)(s,v_{1}) E(s,v_{2}) dv_{2}  dv_{1}ds,
\end{align*}
where
\begin{align*}
	W(s,v_{1},v_{2}) & = w(v_{1}+v_{2})\, \mathrm{sign}(E(s,v_{1}+v_{2})) - w(v_{1}) \,\mathrm{sign}\;(E(s,v_{1})) - w(v_{2}) \,\mathrm{sign}(E(s,v_{2})). 
\end{align*}
Using the properties of the sign function, we get 
\begin{align*}
	W(s,v_{1},v_{2}) E(s,v_{2}) &\leq \left( v_1 + v_2 + (v_1+v_2)^{-\beta} \right) |E(s,v_{2})| + \left( v_1 + v_1^{-\beta} \right) |E(s,v_{2})| \\
	& \quad -  \left( v_2 + v_2^{-\beta} \right) |E(s,v_{2})| \\
	& \leq  \left( 2v_1 + v_1^{-\beta} \right) |E(s,v_{2})| .
\end{align*}
Consequently, 
\begin{align*}
	&\int_{0}^{t}\int_{0}^{\infty}\int_{0}^{\infty} K(v_{1},v_{2}) W(s,v_{1},v_{2}) (c_1+c_2)(s,v_{1}) E(s,v_{2}) dv_{2}  dv_{1}ds \\
	& \quad \leq \int_{0}^{t}\int_{0}^{\infty}\int_{0}^{\infty} K(v_{1},v_{2}) \left( 2v_1 + v_1^{-\beta} \right) (c_1+c_2)(s,v_{1}) |E(s,v_{2})| dv_{2}  dv_{1}ds \\
	&\quad \leq k \int_{0}^{t} \int_{0}^{1} \int_{0}^{1} (v_{1}v_{2})^{-\beta} \left( 2v_1 + v_1^{-\beta} \right) (c_1+c_2)(s,v_{1}) |E(s,v_{2})| dv_{2}  dv_{1}ds \\
	&\quad\quad + 2 k \int_{0}^{t} \int_{0}^{1} \int_{1}^{\infty} v_{1}^{-\beta} v_2 \left( 2v_1 + v_1^{-\beta} \right) (c_1+c_2)(s,v_{1}) |E(s,v_{2})| dv_{2}  dv_{1}ds \\
	& \quad\quad + k \int_{0}^{t} \int_{1}^{\infty} \int_{1}^{\infty} (v_{1}+v_{2}) \left( 2v_1 + v_1^{-\beta} \right) (c_1+c_2)(s,v_{1}) |E(s,v_{2})| dv_{2}  dv_{1}ds \\
	&\quad \leq 6k \int_0^t  \left[ M_{-2\beta}((c_1+c_2)(s)) + M_2((c_1+c_2)(s)) \right] \int_{0}^{\infty} w(v) |E(t,v)| dvds.
\end{align*}
Combining the above inequality with~\eqref{z001} leads us to the following differential inequality for $\nu$:
\begin{equation*}
	\nu(t) \le \nu(0) + (6k+L_0) \int_0^t \left[ 1+ M_{-2\beta}((c_1+c_2)(s)) + M_2((c_1+c_2)(s)) \right]\nu(s) ds , \quad t\in [0,T]. %\label{z013}
\end{equation*}
Recalling~\eqref{z014} and applying Gronwall's lemma gives the stated result.
\end{proof}

\begin{proof}[Proof of \Cref{uniqueness theorem}]   
	Let $c_{1}$ and $c_{2}$ be two weak solutions to \eqref{eq:1.3}--\eqref{eq:1.4} in the sense of  \Cref{uniqueness theorem} corresponding to the initial data $c_{1,0}$ and $c_{2,0}$, respectively. Assume that $c_{1,0} = c_{2,0}$ a.e. in $(0,\infty)$. It readily follows from \Cref{continuousdependence result} that
	\begin{equation*}
		\int_{0}^{\infty} \left( v^{-\beta} + v \right)	|c_{1}(t,v)-c_{2}(t,v)| dv =0,
	\end{equation*}
	which implies that $c_{1}(t)=c_{2}(t)$ for all $t$ in $[0,T]$. This completes the proof of \Cref{uniqueness theorem}.
\end{proof}

%%%%%%%%%%%%%%%%
%%%%%%%%%%%%%%%%
\section*{Acknowledgments}
%%%%%%%%%%%%%%%%
%%%%%%%%%%%%%%%%

The research of SS is supported by Council of Scientific and Industrial Research, India, under the grant agreement No.09/143(0987)/2019-EMR-I. A portion of this research was funded by the Indo-French Centre for Applied Mathematics (MA/IFCAM/19/58) as part of the project Collision-induced breakage and coagulation: dynamics and numerics. Part of this work was performed while PhL enjoyed the hospitality of the Department of Mathematics, Indian Institute of Technology Roorkee.

%%%%%%%%%%%%%%%%
%%%%%%%%%%%%%%%%
\bibliographystyle{siam}
\bibliography{GrowthCoagulationEquation}
%%%%%%%%%%%%%%%%
%%%%%%%%%%%%%%%%

\end{document}